\documentclass[a4paper,11pt,reqno,twoside]{amsart}
\usepackage{graphicx}
\usepackage{epsbox}
\usepackage{subfigure}
\usepackage{color}
\usepackage{cite}
\usepackage{ifpdf}
\usepackage{array}
\usepackage{slashbox}
\usepackage[T1]{fontenc}
\usepackage[latin1]{inputenc}
\usepackage{mathtools}
\usepackage{amsthm,amssymb}
\usepackage{url}

\newcommand*{\C}{\mathbb{C}}
\newcommand*{\R}{\mathbb{R}}
\newcommand*{\Z}{\mathbb{Z}}
\newcommand*{\N}{\mathbb{N}}

\renewcommand*{\geq}{\geqslant}
\renewcommand*{\leq}{\leqslant}
\newtheoremstyle{erdfn}
  {}
  {}
  {\itshape}
  {}
  {\bfseries}
  {}
  { }
  {}
\newtheoremstyle{erthm}
  {}
  {}
  {\itshape}
  {}
  {\bfseries}
  {}
  { }
  {}
\newtheoremstyle{errem}
  {}
  {}
  {}
  {}
  {\bfseries}
  {}
  { }
  {}
\theoremstyle{erthm}
\newtheorem{theorem}{Theorem}[section] 

\newtheorem{proposition}[theorem]{Proposition}
\newtheorem{lemma}[theorem]{Lemma}

\theoremstyle{erdfn}

\theoremstyle{errem}

\newtheorem{remark}{Remark}

\numberwithin{equation}{section}
\title[A canonical system of differential equations ]%
      {A canonical system of differential equations \\ arising from the Riemann zeta-function} 
\author[M. Suzuki]{Masatoshi Suzuki}

\date{Version of \today}
\address{Department of Mathematics, Tokyo Institute of Technology, 2-12-1 Ookayama, Meguro-ku, Tokyo 152-8551, Japan}
\email{msuzuki@math.titech.ac.jp}
\subjclass[2000]{}
\AtBeginDocument{%
\mathtoolsset{showonlyrefs,mathic = true}

\begin{abstract}
This paper has two main results, which relate to a criteria for the Riemann hypothesis 
via the family of functions $\Theta_\omega(z)=\xi(\frac{1}{2}-\omega-iz)/\xi(\frac{1}{2}+\omega-iz)$, 
where $\omega>0$ is a real parameter and $\xi(s)$ is the Riemann xi-function. 
The first main result is necessary and
sufficient conditions for 
$\Theta_\omega$ to be a meromorphic inner function in the upper half-plane. 
It is related to the Riemann hypothesis directly whether $\Theta_\omega$ is a meromorphic inner function. 
In comparison with this, 
a relation of the Riemann hypothesis and the second main result is indirect.
It relates to the theory of de Branges, 
which associates a meromorphic inner function and a canonical system of linear differential equations 
(in the sense of de Branges). 
As the second main result, the canonical system associated with $\Theta_\omega$ is constructed explicitly 
and unconditionally under the restriction of the parameter $\omega >1$  
by applying a method of J.-F. Burnol in his recent work on the gamma function 
to the Riemann xi-function. 
If such construction is extended to all $\omega > 0$ unconditionally,  
we get a criterion for the Riemann hypothesis in terms of a family of canonical systems 
parametrized by $\omega>0$, which explains the validity of the Riemann hypothesis 
as positive semidefiniteness of the corresponding family of Hamiltonian matrices. 
\end{abstract}
\maketitle
}
\begin{document}

%
%
\section{Introduction} 
%
%

Let $\zeta(s)$ be the Riemann zeta function. 
The set of all non-trivial zeros of the Riemann zeta function 
coincides with the set of all zeros of the Riemann xi-function 
\begin{equation*}
\xi(s)=\frac{1}{2} s(s-1) \, \pi^{-s/2}\Gamma\left(\frac{s}{2}\right) \, \zeta(s).
\end{equation*} 
The Riemann hypothesis, which is often abbreviated to RH, assert that 
all zeros of $\xi(s)$ lie on the critical line $\Re(s)=1/2$. 
We attempt to understand the nontrivial zeros of the Riemann zeta function 
via the family of functions 
\begin{equation}\label{AB}
A^\omega(z)  := \frac{1}{2}(\xi(s+\omega)+\xi(s-\omega)), \quad
B^\omega(z)  := \frac{i}{2}(\xi(s+\omega)-\xi(s-\omega)), 
\end{equation}
where $s=1/2-iz$ and $\omega$ is a positive real parameter. 
Functions $A^\omega(z)$ and $B^\omega(z)$ take real values on the real line 
and satisfy the functional equations $A^\omega(z)=A^\omega(-z)$ and $B^\omega(z)=-B^\omega(-z)$ 
by the functional equations $\xi(s)=\xi(1-s)$ and $\xi(s)=\overline{\xi(\bar{s})}$. 

If all zeros of $A^\omega(z)$ lie on the real line for every $\omega>0$, 
it implies RH by Hurwitz's theorem in complex analysis. 
Conversely, 
all zeros of $A^\omega(z)$ lie on the real line 
for $\omega \geq 1/2$ unconditionally 
and for $0<\omega<1/2$ under RH 
by a result of Lagarias ~\cite{MR2187786} 
(see also Li~\cite{MR2496465} for an unconditional result for $0<\omega<1/2$). 
We abbreviate to ${\rm RH}(A^\omega)$ (resp.  ${\rm RH}(B^\omega)$) the assertion that 
all zeros of $A^\omega(z)$ (resp.  $B^\omega(z)$) lie on the real line, 
and abbreviate ${\rm RH}(A^\omega)$ and ${\rm RH}(B^\omega)$ as ${\rm RH}(A^\omega,B^\omega)$ . 
Then the above things are stated as follows: 

\begin{proposition}\label{prop_basic}
{\rm RH} holds if and only if ${\rm RH}(A^\omega)$ holds for all $\omega > 0$. 
\end{proposition}

\noindent
The latter condition is easier to study in that it is currently known to hold for all
$\omega \geq 1/2$.  Also it is known to be related to some operators. 
We will study the latter problem of finding linear differential equation systems with boundary conditions for
which the zeros of $A^\omega(z)$ are eigenvalues, for a suitable range of $\omega$. 
\medskip

It is believed that a promising way to prove RH is the Hilbert-P\'olya conjecture 
which asserts that the non-trivial zeros of the Riemann zeta function correspond to 
eigenvalues of some positive operator if RH is true. 
Therefore, if we refer to Proposition \ref{prop_basic}, it is an interesting problem to find a canonical way realizing the zeros of $A^\omega(z)$ as 
the eigenvalues of some positive operator. 
Fortunately, as shown in \cite{MR2187786} (see also \cite{MR2261101}), 
it is possible for $\omega \geq 1/2$ unconditionally and for $0<\omega<1/2$ under RH 
if we use the theory of de Branges spaces 
that are kind of reproducing kernel Hilbert spaces consisting of entire functions. 
However, unfortunately, RH is used essentially in \cite{MR2187786}  
to construct corresponding de Branges spaces for $0<\omega<1/2$. 

According to a general theory of de Branges spaces, 
there exists a unique canonical system of linear differential equations 
associated with a given de Branges space up to a normalization. 
And also, it is known that a special class of canonical system is transformed 
into a pair of Schr\"odinger equations endowed with a pair of (distributional) potentials. 
At this stage, the validity of ${\rm RH}(A^\omega)$ is encoded in analytic properties of potentials (see \cite{MR2261101}, and also \cite{MR2551895}). 
Hence, a possible way to avoid assuming RH 
in the construction of the de Branges space arising from $A^\omega(z)$ for $0<\omega<1/2$ 
is a direct construction of a pair of potentials without RH. 
However, in general, it is difficult to determine a pair of potentials corresponding to a given de Branges space, 
and it is so for the de Branges space arising from $A^\omega(z)$ even if $\omega \geq 1/2$. 

A goal of the present paper is to describe unconditionally for $\omega > 1$, a canonical system 
and corresponding pair of potentials associated with a de Branges space arising from $A^\omega(z)$ 
in terms of Fredholm determinants of certain compact integral operators (Theorem \ref{thm_4}). 
The restriction $\omega>1$ is expected to be relaxed to $\omega>0$ if RH is true (see comments after Theorem \ref{thm_4} and Section 5 for details). 

In order to explain the above things more precisely, 
we review results on de Branges spaces, canonical systems and model subspaces.

%
%
\subsection{de Branges spaces and canonical system} 
%
%
At first, we review the theory of de Branges spaces according to de Branges~\cite{MR0229011} and Lagarias~\cite{MR2261101,MR2551895} 
(see also Remling~\cite{MR1943095}). 
Let $E$ be an entire function satisfying the Hermite-Biehler condition
\begin{equation}\label{101}
|E(z)| > |E^\sharp(z)| \quad \text{for} \quad \Im(z)>0, 
\end{equation}
where $E^\sharp(z)=\overline{E(\bar{z})}$. 
Then entire function $E$ generates the de Branges space
\begin{equation*}
B(E) := \{ f ~|~ \text{$f$ is entire, $f/E$ and $f^\sharp/E \in  H^2$} \}
\end{equation*}
endowed with norm $\Vert f \Vert_{B(E)} := \Vert f/E \Vert_{L^2(\R)}$, 
where $H^2 = H^2(\C^+)$ is the Hardy space in the upper half-plane $\C^+$ 
which is defined to be the space of all analytic functions $f$ in $\C^+$ endowed with 
norm $\Vert f \Vert_{H^2}^2 := \sup_{v>0} \int_{\R} |f(u+iv)|^2 \, du < \infty$. 
An entire function $F(z)$ is called a {\it real entire function} if $F(z)=F^\sharp(z)\,(:=\overline{F(\bar{z})})$. 
Condition \eqref{101} implies that real entire functions 
\begin{equation*}
A(z) := \frac{1}{2}(E(z)+E^\sharp(z)), \qquad 
B(z) := \frac{i}{2}(E(z)-E^\sharp(z)) 
\end{equation*}
have real zeros only, and these zeros interlace. 
Moreover, if $E(z) \not=0$ on the real line, all zeros are simple (\cite[Lemma 5]{MR0114002}). 
A de Branges space $B(E)$ has an unbounded operator $({\mathsf M},{\frak D}({\mathsf M}))$, 
multiplication by the independent variable $({\mathsf M}f)(z)=zf(z)$ with the domain 
${\frak D}({\mathsf M})=\{f\in B(E) \,|\, zf(z) \in B(E)\}$. 
The multiplication operator ${\mathsf M}$ is symmetric and closed,  
and if ${\frak D}({\mathsf M})$ is dense in $B(E)$, it has deficiency indices $(1,1)$, 
and hence  has a family of self-adjoint extensions ${\mathsf M}_\theta$ parametrized by $\theta \in [0,\pi)$. 
In particular, ${\mathsf M}_{\pi/2}$ and ${\mathsf M}_{0}$ have pure discrete spectrum 
located at zeros of $A(z)$ and $B(z)$ respectively. 

We put the normalization $E(0)=1$ for entire functions $E$ satisfying \eqref{101} for a convenience. 
Then, for a given de Branges space $B(E)$, 
there exists a chain of de Branges spaces $B(E_a) \subset B(E)$, $0<a \leq c \,(\leq \infty)$,    
endowed with a family of entire functions $E_a(z)$ satisfying \eqref{101} and $E_a(0)=1$ 
such that $B(E_a) \subset B(E_{a'})$ for $a < a'$, 
and the parametrized pair of real entire functions $(A_a,B_a):=(\frac{1}{2}(E_a+E_a^\sharp),\frac{i}{2}(E_a-E_a^\sharp))$ satisfies 
the {\it canonical system} 
\begin{equation*}
\frac{\partial}{\partial a} \begin{bmatrix} A_a(z) \\ B_a(z) \end{bmatrix}  
= z \begin{bmatrix} 0 & -1 \\ 1 & 0 \end{bmatrix} H(a) \begin{bmatrix} A_a(z) \\ B_a(z) \end{bmatrix}, 
\quad 
H(a) = \begin{bmatrix} \alpha(a) & \beta(a) \\ \beta(a) & \gamma(a) \end{bmatrix} 
\end{equation*}
of linear differential equations with the initial condition 
\[
\lim_{a \to 0^+}(A_a(z),B_a(z)) = (1,0)
\] 
for each $z \in \C$, and $E_c(z)=E(z)$ 
(see \cite[Theorem 40]{MR0229011}, but note that it is formulated in terms of integral equations).  
Here the matrix $H(a)$ is a {\it measurable} and {\it real positive semidefinite symmetric} matrix for almost all $0<a \leq c$, 
and which is integrable over the interval.  
The matrix $H(a)$ is often called a Hamiltonian of a canonical system. 
These properties of $H(a)$ are crucial, 
because the initial function $E$ can be recovered from $H(a)$ 
by solving the canonical system with the above initial condition 
(\!\!\cite[Theorem 41]{MR0229011}). 
On the other hand, the spectrum of the extended multiplication operator ${\mathsf M}_\theta$ coincides with the spectrum of 
the above canonical system with the boundary condition 
$\lim_{a \to 0^+}(A_a(z),B_a(z)) = (1,0)$ and $A_c(z)\sin \theta - B_c(z)\cos \theta=0$. 

If $H(a)$ is diagonal ($\beta(a)=0$) and $\alpha(a)\gamma(a)=1$ almost everywhere in $(0,c]$, 
the corresponding canonical system is transformed into a pair of Schr\"odinger equations
\begin{equation*}
\left( - \frac{d^2}{da^2} + V^\pm (a) \right)\psi(a,z) = z^2 \psi(a,z), 
\quad
V^\pm(a) = \frac{1}{4}\left(\frac{\alpha'(a)}{\alpha(a)} \right)^2 \pm \frac{1}{2}\left(\frac{\alpha'(a)}{\alpha(a)} \right)^\prime,
\end{equation*} 
and the initial $E$ is recovered by solving the pair of Schr\"odinger equations 
under the corresponding initial conditions. 

Eventually, condition \eqref{101} of $E$ is encoded in analytic properties of $H(a)$ or $V^\pm(a)$. 
In general, it is difficult to determine $H(a)$ or $V^\pm(a)$ for given $E$ 
except for few special examples (see Chapter 3 of \cite{MR0229011}, 
and also \cite{MR2551895, B1}). 
%
%
\subsection{Spectral realization of zeros of $A^\omega$ and $B^\omega$} 
%
%
Suppose that the condition 
\begin{equation}\label{102}
|\xi(s+\omega)|>|\xi(s-\omega)| \quad \text{for} \quad \Re(s)>\frac{1}{2}
\end{equation}
holds. Then we find that $E(z)=E^\omega(z):=\xi(\frac{1}{2}+\omega-iz)$ satisfies \eqref{101} 
by using the functional equations $\xi(s)=\xi(1-s)$ and $\xi(s)=\overline{\xi(\bar{s})}$.  
Thus the de Branges space $B(E^\omega)$ is defined, and ${\rm RH}(A^\omega, B^\omega)$ holds. 
By a result of \cite{MR2187786}, condition \eqref{102} holds for $\omega \geq 1/2$ unconditionally 
and for $0<\omega<1/2$ under RH. 
This is the reason why RH implies ${\rm RH}(A^\omega)$ for all $\omega>0$. 
However, for fixed $\omega>0$, 
condition \eqref{102} is only a sufficient condition to ${\rm RH}(A^\omega, B^\omega)$, 
that is, ${\rm RH}(A^\omega)$ or ${\rm RH}(B^\omega)$ may be true even if condition \eqref{102} does not hold. 

Anyway, we can regard the zeros of $A^\omega(z)$ and $B^\omega(z)$ as discrete spectrum of 
self-adjoint extensions of $({\mathsf M},{\frak D}({\mathsf M}))$ on $B(E^\omega)$ 
for $\omega \geq 1/2$ unconditionally and for $0<\omega<1/2$ under RH. 
Therefore, a natural problem on ${\rm RH}(A^\omega)$ and a spectral realization of the zeros of $A^\omega(z)$  
is to find a way avoiding RH for $0<\omega<1/2$. 
A possible approach is to construct $H(a)$ or $V^\pm(a)$ associated with $B(E^\omega)$ 
without assuming RH, and recover $E^\omega$, $A^\omega$ and $B^\omega$ 
from the canonical system attached to $H(a)$ or 
the pair of Schr\"odinger equations attached to $V^\pm(a)$. 
We attempt to follow this way by using the theory of model subspaces.  

%
%
\subsection{Model subspaces} 
%
%
For further discussions, we review a theory of model spaces according to Havin--Mashreghi~\cite{MR2016246, MR2016247} 
(see also Baranov~\cite{MR1855436}, Makarov--Poltoratski~\cite{MR2215727}). 
A function $\Theta$ is called an {\it inner function} in $\C^+$ 
if it is a bounded analytic function in $\C^+$ 
such that $\lim_{v \to 0^+} |\Theta(u + iv)| = 1$ 
for almost all $u \in \R$ with respect to Lebesgue measure.  
If an inner function $\Theta$ in $\C^+$ is extended to a meromorphic function in $\C$, 
it is called a {\it meromorphic inner function} in $\C^+$. 
It is known that every meromorphic inner function is expressed as 
$\Theta = E^\sharp/E$ by using an entire function $E$ 
satisfying \eqref{101}. 
For an inner function $\Theta$, 
a {\it model subspace} (or coinvariant subspace) $K(\Theta)$ is defined by the orthogonal complement 
\begin{equation}\label{103}
K(\Theta)=H^2 \ominus \Theta H^2,
\end{equation}
where $\Theta H^2 = \{ \Theta(z)F(z) \, |\, F \in H^2\}$. 
%
%
%
It has the alternative representation 
\begin{equation}\label{104}
K(\Theta) = H^2 \cap \Theta \bar{H}^2,
\end{equation}
where $\bar{H}^2 = H^2(\C^-)$ is the Hardy space in the lower half-plane $\C^-$. 
If $\Theta$ is a meromorphic inner function such that $\Theta=E^\sharp/E$, 
the model subspace $K(\Theta)$ is isomorphic to the de Branges space $B(E)$ as a Hilbert space 
by $K(\Theta) \to B(E): \, f \mapsto fE$. 
In particular, $K(\Theta)$ is a reproducing kernel Hilbert space. 
The reproducing kernel of $K(\Theta)$ is given by 
\begin{equation} \label{rp_1}
K(z,w) = \frac{1}{2\pi i} \, \frac{1 - \overline{\Theta(z)}\Theta(w)}{\bar{z} - w} \quad (z,w \in \C^+), 
\end{equation}
and the reproducing formula $f(z)=\langle f, K(z,\cdot)\rangle_{L^2(\R)}\, (f \in K(\Theta),~z \in \C^+)$ 
remains true for $z \in  \R$ if $\Theta$ is analytic in a neighborhood of $u$, 
where $\langle f,g \rangle_{L^2(\R)} = \int_{\R} f(u)\overline{g(u)}du$. 
%
%
\subsection{Model subspaces related to $A^\omega$ and $B^\omega$} 
%
%
Now we apply the theory of model subspaces to the spaces $B(E^\omega)$ of Section 1.2. 
For positive real $\omega$, we define the meromorphic function $\Theta_\omega(z)$ in $\C$ by 
\begin{equation} \label{105_1}
\Theta_\omega(z) := \frac{\xi(\frac{1}{2}-\omega-iz)}{\xi(\frac{1}{2}+\omega-iz)} . 
\end{equation}
%
Then we have 
\begin{equation}\label{105}
\Theta_\omega(z)\Theta_{\omega}(-z)=1 \quad \text{for} \quad z \in \C,  
\end{equation}
\begin{equation}\label{106}
|\Theta_\omega(u)|=1 \quad \text{for} \quad u \in \R,  
\end{equation}
\begin{equation}\label{107}
\Theta_\omega(0)=1, 
\end{equation}
by functional equations $\xi(s)=\xi(1-s)$ and $\xi(\bar{s})=\overline{\xi(s)}$. 

The inequality \eqref{102} can now be reinterpreted as the condition
\begin{equation}\label{108}
|\Theta_\omega(z)|<1 \quad \text{for} \quad \Im(z)>0
\end{equation}
and vice versa. 
Recall that condition \eqref{102} is known to hold for $\omega \geq 1/2$ unconditionally 
and for $0<\omega<1/2$ under RH. By \eqref{106}, when condition \eqref{108} holds, 
it implies that $\Theta_\omega(z)$ is a meromorphic inner function in $\C^+$. 
Therefore, whenever \eqref{108} holds, 
we obtain a model subspace $K(\Theta_\omega)$ 
which is isomorphic to the de Branges space $B(E^\omega)$ generated by $E^\omega(z)=\xi(\frac{1}{2}+\omega-iz)$. 
Here we mention the following equivalence relation.
\begin{proposition}  \label{intro_2}
Let $\omega_0 \geq 0$. Then the following are equivalent: 
\begin{enumerate}
\item $\zeta(s)\not=0$ for $\Re(s)>\frac{1}{2}+\omega_0$, 
\item $\Theta_\omega(z)$ is a meromorphic inner function in $\C^+$ for every $\omega>\omega_0$. 
\end{enumerate}
\end{proposition}
\begin{proof}
Assume that $0 \leq \omega_0 <1/2$ since we have nothing to say for $\omega \geq 1/2$. 
By applying Theorem 4 of \cite{MR2220265}, we find that (1) implies that \eqref{108} holds for every $\omega>\omega_0$ . 
Thus we obtain (1)$\Rightarrow $(2).  
The converse implication (2)$\Rightarrow$(1) is proved by a way similar to the proof of Theorem 2.3 (1) in \cite{Su}.
\end{proof}

The changing of consideration from $B(E^\omega)$ to $K(\Theta_\omega)$ 
has the advantage that spaces $\Theta_\omega H^2$, $\Theta_\omega \bar{H}^2$, $H^2 \ominus (H^2 \cap \Theta_\omega H^2)$ and $H^2 \cap \Theta_\omega \bar{H}^2$ 
are defined even if $\Theta_\omega(z)$ is not necessarily a meromorphic inner function in $\C^+$  (see \eqref{103} and \eqref{104}), 
and it allows us to study these spaces for the range $0 < \omega < 1/2$ without assuming RH. 
(Note that $\Theta H^2 \not\subset H^2$ in general if $\Theta$ is not necessary a inner function in $\C^+$.) 
To make a further discussion, we use Fourier analysis. 
%
%
\subsection{An operator related to $K(\Theta_\omega)$} 
%
%
As usual we identify $H^2$ and $\bar{H}^2$ with subspaces of $L^2(\R)=L^2((-\infty,\infty),du)$ 
via nontangential boundary values on the real line such that $L^2(\R)=H^2 \oplus \bar{H}^2$. 
Then the shifted Fourier transform  
\begin{equation*}
\aligned
{\mathsf F}_{1/2}:L^2((0,\infty),dx) \to L^2(\R): \quad 
({\mathsf F}_{1/2} f)(z) 
& = \int_{0}^{\infty} f(x) \, x^{\frac{1}{2}+iz} \, \frac{dx}{x},  \\
{\mathsf F}_{1/2}^{-1}: L^2(\R) \to L^2((0,\infty),dx): \quad 
({\mathsf F}_{1/2}^{-1} g)(z) 
& = \frac{1}{2\pi}\int_{-\infty}^{\infty} g(u) \, x^{-\frac{1}{2}-iu} \, du
\endaligned
\end{equation*}
provides an isometry of $L^2$-spaces up to a constant 
such that $H^2 = {\mathsf F}_{1/2}L^2((1,\infty),dx)$ and $\bar{H}^2 = {\mathsf F}_{1/2}L^2((0,1),dx)$ by the Paley-Wiener theorem. 

Fourier analysis on $K(\Theta_\omega)$ and $\Theta_\omega H^2$ enables us to state equivalent or sufficient conditions 
that $\Theta_\omega(z)$ is a meromorphic inner function in $\C^+$ (Theorem \ref{thm_2}). 

On the other hand, condition \eqref{106} allows us to define the Hankel type operator
\begin{equation*}
({\mathsf H}_\omega^\ast f)(x) = \int_{0}^{\infty} h_\omega^\ast(xy) \, f(y) \, dy 
\end{equation*}
on $L^2((0,\infty),dx)$ endowed with the kernel given by 
\begin{equation}\label{109}
h_\omega^\ast(x) = \frac{1}{2\pi} \int_{-\infty}^{\infty}  \Theta_\omega(u) \, x^{-\frac{1}{2}-iu} \, du. 
\end{equation}
Of course the definition of ${\mathsf H}_\omega^\ast$ has only a formal sense 
because of the problem of the convergence of integral in \eqref{109}. 
However $h_\omega^\ast(x)$ is going to be identified with the function $h_\omega(x)$ in Section 2,  
and then ${\mathsf H}_\omega^\ast$ is going to be justified  
as the operator ${\mathsf H}_\omega$ obtained by replacing the kernel $h_\omega^\ast(x)$ by $h_\omega(x)$.  
Moreover the operator ${\mathsf H}_\omega$ is extended to an isometry from $L^2((0,\infty),dx)$ to $L^2((0,\infty),dx)$
for $\omega \geq 1/2$ unconditionally, and for $0<\omega<1/2$ under RH (see Lemma \ref{lem_401}). 

As developed in Burnol~\cite{B1} 
(and his other related works ~\cite{MR2096473, MR2267058, MR2310951}), 
the Hankel type operator ${\mathsf H}_\omega$ and its kernel $h_\omega(xy)$ is quite useful 
to study a structure of subspaces of ${\mathsf F}_{1/2}^{-1}K(\Theta_\omega)$ 
corresponding to de Branges subspaces of $B(E^\omega) \simeq K(\Theta_\omega)$. 
By applying Burnol's theory to ${\mathsf H}_\omega$ and $h_\omega(x)$, 
we derive a canonical system of $B(E^\omega)$ under the restriction $\omega>1$ (Theorem \ref{thm_4} and studying in Section 4).  
Recall that the structure of subspaces of a de Branges space is controlled by its canonical system. 

%
%
\subsection{Summary of issues} 
%
%
Briefly, we have two issues. 
The first is to state a (nice) criterion for the innerness of $\Theta_\omega(z)$. 
It is directly related to the zero-free region of $\zeta(s)$ (Proposition \ref{intro_2}). 
The second is to describe the Hamiltonian $H_\omega(a)$ of the canonical system of $B(E^\omega)$ explicitly 
by assuming that $\Theta_\omega(z)$ is a meromorphic inner function in $\C^+$ if $0<\omega<1/2$. 
If it is done, we can state that $\Theta_\omega(z)$ is a meromorphic inner function in $\C^+$ 
if and only if $(A^\omega,B^\omega)=(A_c,B_c)$ for the solution $(A_a,B_a)$ of the canonical system for $H_\omega(a)$ on $a \in (0,c]$ 
satisfying $\lim_{a \to 0^+} (A_a,B_a)=(E^\omega(0),0)$. 
This description explains the innerness of $\Theta_\omega(z)$ as a consequence of properties of $H_\omega(a)$, 
and it provides a criterion for a zero-free region of $\zeta(s)$ 
in terms of a family of canonical systems attached to $\{H_\omega(a)\}_{\omega>\omega_0}$ via Proposition \ref{intro_2}. 

However the second problem is not trivial even if $\omega \geq 1/2$. 
In this paper, we deal with the case $\omega>1$ for the second problem as the first attempt.  
%
%
\subsection{Organization of the paper} 
%
%
The paper is organized as follows. 
In Section 2, we state main results Theorem \ref{thm_2} and Theorem \ref{thm_4} 
after a small preparation of notation. 
The first one is equivalent conditions on the Hermite-Biehler condition \eqref{108} 
in terms of the function $h_\omega(x)$ for fixed $\omega>0$. This is proved in Section 3. 
The second one is a result on the canonical system of $B(E^\omega) \simeq K(\Theta_\omega)$
under the restriction $\omega>1$. 
It is proved in Section 4 together with related studies and auxiliary results. 
In addition, we present more sufficient or equivalent conditions that $\Theta_\omega(z)$ 
is a meromorphic inner function in $\C^+$ 
in Appendix A (Theorem \ref{thm_3}). 
\medskip

Here we mention that this paper, particularly Appendix A, is a sequel to \cite{Su}, though it is independent and can be read separately. 
The operator ${\mathsf H}_\omega^\ast$ of Section 1.5 is also justified as the Watson transform: 
\begin{equation*}
({\mathsf H}_\omega^{\ast\ast} f)(x) = \frac{d}{dx}\int_{0}^{\infty} h_\omega^{\ast\ast}(xy) \, f(y) \, \frac{dy}{y}, 
\qquad 
h_\omega^{\ast\ast}(x) = \frac{1}{2\pi} \int_{-\infty}^{\infty} \frac{\Theta_\omega(u)}{\frac{1}{2}-iu} \, x^{\frac{1}{2}-iu} \, du, 
\end{equation*}
which gives a linear involution on $L^2((0,\infty),dx)$ under \eqref{105} (only for real $z$) and \eqref{106} 
(see Titchmarsh \cite[\S8.5]{MR942661},  Bochner--Chandrasekharan~\cite[Chap.V, \S2]{MR0031582}). 
Moreover, ${\mathsf H}_\omega^{\ast\ast} = {\mathsf H}_\omega$ if $\Theta_\omega(z)$ is inner in $\C^+$. 
The Watson transform has the advantage that $h_\omega^{\ast\ast}(x)$ always exists in $L^2$-sense by \eqref{106}, 
and belongs to $L^2((0,\infty),dx)$. 
While the modified function 
\[
h_\omega^{\langle 1 \rangle}(x) = \frac{1}{2\pi} \int_{-\infty}^{\infty} \frac{\Theta_\omega(u)}{-iu} \, x^{\frac{1}{2}-iu} \, du
\] 
does not belong to $L^2((0,\infty),dx)$ although it is justified as a function (Appendix A).  
However it is also useful to study the space $K(\Theta_\omega)$ and the operator ${\mathsf H}_\omega$ 
because of formula \eqref{rp_1} for the reproducing kernel. 
In fact, several sufficient or equivalent conditions that 
$\Theta_\omega(z)$ is inner in $\C^+$ are stated in terms of $h_\omega^{\langle 1 \rangle}(x)$ 
(Theorem \ref{thm_3}) as well as Theorem \ref{thm_2}. 
Moreover, if $\Theta_\omega(z)$ is inner in $\C^+$,  we obtain 
\[
({\mathsf H}_\omega f)(x)
 = \int_{0}^{\infty} h_\omega(xy) \, f(y) \, dy 
 = \sqrt{x} \frac{d}{dx}\sqrt{x} \int_{0}^{\infty}  h_\omega^{\langle 1 \rangle}(xy) \, f(y) \, dy
\]
for compactly supported smooth functions $f$, and it is extended to $L^2((0,\infty),dx)$ (Theorem \ref{thm_5}). 
The function $h_\omega^{\langle 1 \rangle}(x)$ was introduced and studied in \cite{Su} 
for more general $L$-functions, 
but a relation with spaces $B(E^\omega) \simeq K(\Theta_\omega)$ and operators ${\mathsf H}_\omega$ 
were not mentioned there. In this sense, this paper is a sequel to \cite{Su}. 

%
%
\subsection{De Branges' works} 
%
%
Finally, we comment on de Branges' works on $B(E^\omega)$. 
The de Branges space $B(E^\omega)$ was considered first  
for the special value $\omega=1/2$ in de Branges~\cite[pp.10--14]{MR838785}, 
motivating to generalize the Lax-Phillips scattering theory to the Laplace-Beltrami operator, 
and for $\omega \geq 1/2$ in the subsequent paper \cite[pp.205--210]{MR1165869}. 
(Precisely, we need to replace $\zeta(s)$ 
by a Dirichlet $L$-function $L(s,\chi)$ attached to an even primitive Dirichlet character $\chi$ in \cite{MR838785}). 
De Branges gave a sufficient condition on $B(E)$ attached to general entire function $E$ 
satisfying \eqref{101} such that the zeros of $E(z)$ lie on the line  $\Im(z) = -1/2$,  
which implies the (generalized) RH when $E=E^{\omega}$ for $\omega=1/2$. 
However Conrey and Li ~\cite{MR1792282} showed that $B(E^{\omega})$ $(\omega=1/2)$ does not satisfy de Branges' condition. 
For $\omega \geq 1/2$ de Branges studied the space $B(E^\omega)$ by associating it with the weighted Hardy space ${\mathcal F}(W)=WH^2$ 
for the  weight function $W(z)=\frac{1}{4}(s+\omega)(s+\omega -1)\Gamma(\frac{s+\omega}{2})$ 
with $s=\frac{1}{2}-iz$, 
but we omit the details of this topic (see \cite{MR1165869}, and also \cite{MR1792282}).  

In any case, de Branges directly related RH with a condition on $B(E^\omega)$ for fixed $\omega \geq 1/2$. 
On the other hand, we reduced RH to the family of spaces $\{B(E^\omega)\}_{\omega>0}$, 
and study each space $B(E^\omega)$ depending on a level of difficulty, 
which is determined by the value $\omega$. 
This is a major difference with de Branges' approach and ours. 
\medskip

\noindent
{\bf Acknowlegements}~
I heartily thank the reviewer for many detailed and helpful comments and corrections. 
In particular, the readability of the paper was quite improved, 
and an error of the proof of Lemma \ref{lem_403} in the initial version was corrected by comments of the reviewer. 
This work was supported by KAKENHI (Grant-in-Aid for Young Scientists (B)) No. 21740004. 
%
%
\section{Main Results}
%
%
Our first result is to derive an expression for $\Theta_\omega(z)$ as a Mellin transform 
of a function $h_\omega(x)$ defined for $0 < x < \infty$, 
which is valid for all real $\omega > 0$ (Proposition \ref{prop_1}). 
To define this function we first define the numbers 
\begin{equation} \label{203}
c_\omega(n) := n^{\omega}\sum_{d|n} \frac{\mu(d)}{d^{2\omega}} = n^{\omega}\prod_{p|n}\left(1-\frac{1}{p^{2\omega}}\right) 
\end{equation}
for natural numbers $n$, where $\mu(n)$ is the M\"obius function, that is,   
$\mu(n)=0$ if $n$ is not a square free number, and $\mu(n)=(-1)^k$ if $n$ is a product of $k$ distinct primes. 
The arithmetic function $n \mapsto J_{2\omega}(n):=n^{\omega}c_\omega(n)$ is called Jordan's totient function,  
which gives Euler's totient function $\varphi(n)$ for $\omega=1/2$. 

Next we introduce a function $g_\omega(x)$ defined on $(0,\infty)$ by
%
\begin{equation*}
\aligned
g_\omega(x)
 = \frac{2\pi^\omega}{\Gamma(\omega)}\left( x^{2-\omega} (1-x^2)^{\omega-1}
- \omega x^{\omega-1} 
\int_{x^2}^{1} t^{\frac{1}{2} - \omega}(1-t)^{\omega-1} \, dt
\right)
\endaligned
\end{equation*}
for $0<x< 1$, and $g_\omega(x)=0$ for $x>1$. 
It is continuous on $(0,1)$ and $(1,\infty)$. 
The behavior of $g_\omega$ near $x=1$ and $x=0$ is as follows. 
We have 
\begin{equation}\label{201}
g_\omega(x) 
=\frac{(2\pi)^\omega}{\Gamma(\omega)} 
(1-x)^{\omega-1}+o(1)  \quad \text{as} \quad x \to 1^-.  
\end{equation}
Therefore $g_\omega$ is continuous at $x=1$ if and only if $\omega>1$, 
and it is $L^1$ (resp. $L^2$) at $x=1$ if $\omega>0$ (resp. $\omega>1/2$). 
On the other hand, we have
\begin{equation*}
g_\omega(x) 
=
\begin{cases}
-4 \,\omega \pi^{\omega - 1/2} \Gamma(3/2 - \omega) \, x^{\omega-1} +o(1), & 0<\omega< 3/2, \\
~4 \,\pi \sqrt{x} \, (3 \log x + 4 - 3 \log 2) + o(1), & \omega=3/2, \\
-6 \pi^\omega (2 \omega - 3)^{-1} \Gamma(\omega)^{-1} \, x^{2-\omega} + o(1), & \omega > 3/2,
\end{cases} \quad \text{as} \quad x \to 0^+.
\end{equation*}
Thus $g_\omega$ is $L^1$ (resp. $L^2$) at $x=0$ if $0<\omega<3$ (resp. $1/2<\omega<5/2$). 
The size of the singularity at $x = 1$ will be important in the sequel 
because it influences the type of operators ${\mathsf H}_{\omega,a}$ below, 
while there is no need to be careful about the behavior around $x=0$ in this paper. 

Finally, we define the real-valued function $h_\omega$ on $(0,\infty)$ by 
\begin{equation} \label{204}
h_\omega(x) = 
\displaystyle{\frac{1}{x} \sum_{n=1}^{\lfloor x \rfloor} c_\omega(n) \, g_\omega\left(\frac{n}{x}\right)}
\end{equation}
for $x>1$, and $h_\omega(x)=0$ for $0<x<1$. 
The value $h_\omega(1)$ may be undefined, 
since $c_\omega(1)=1$ and $g_\omega(1^-) = +\infty$ for $0<\omega < 1$ by \eqref{201}.
By definition, $h_\omega$ has a support in $[1,\infty)$, 
and is $L^1$ (resp. $L^2$) on every finite interval $[1,b]$ if $\omega>0$ (resp. $\omega>1/2$).  
On the other hand, the behavior of $h_\omega$ at $x=+\infty$ is not obvious from its definition 
(see \eqref{est_h} below). Now the first result is stated as follows. 
\begin{proposition} \label{prop_1} For $\omega>0$ and $\Im(z)>1/2+\omega$, we have 
\begin{equation} \label{206}
\int_{0}^{\infty} h_{\omega}(x) \, x^{\frac{1}{2}+iz} \, \frac{dx}{x} 
 = \Theta_\omega(z),
\end{equation}
where the integral converges absolutely. 
\end{proposition}

We introduce more notation in order to sate the main results mentioned in the introduction. 
By \eqref{106}, $F(z) \mapsto \Theta_\omega(z) F(z)$ defines a map  $L^2(\R) \to L^2(\R)$. 
We denote it also by $\Theta_\omega$ if no confusion arises, 
and define 
\begin{equation*}
\widehat{\Theta}_\omega={\mathsf F}_{1/2}^{-1}\Theta_\omega{\mathsf F}_{1/2}:L^2((0,\infty),dx) ~ \to~ L^2((0,\infty),dx).
\end{equation*} 
If $\Theta_\omega(z)$ is an inner function in $\C^+$, 
images $\Theta_\omega H^2$ and $\widehat{\Theta}_\omega L^2((1,\infty),ds)$ are subspaces of $H^2$ and $L^2((1,\infty),dx)$, respectively. 
Obviously the map $\widehat{\Theta}_{\omega}$ is related to the function $h_\omega$ by \eqref{206}. 
In fact the innerness of $\Theta_\omega(z)$ is described in terms of $h_{\omega}$ as follows. 
\begin{theorem} \label{thm_2}
Let $\omega>0$. The function $\Theta_\omega(z)$ is a meromorphic inner function in $\C^+$ 
if and only if one of the following conditions holds: 
\begin{enumerate}
\item $ \widehat{\Theta}_\omega f = h_\omega \ast f $ for every $f \in L^2((1,\infty),dx)$, 
where 
\begin{equation*}
(h_\omega \ast f)(x) = \int_{0}^{\infty} h_\omega(x/y) f(y) \frac{dy}{y}.
\end{equation*}
\item $\widehat{\Theta}_\omega f$ vanishes on $(0,1)$ for every $f \in L^2((1,\infty),dx)$. 
\item $h_\omega \ast f$ belongs to $L^2((0,\infty),dx)$ for every $f \in L^2((1,\infty),dx)$. 
\end{enumerate}
\end{theorem} 
Suppose that $\Theta_\omega(z)$ is an inner function in $\C^+$. 
Then 
\begin{equation} \label{208}
({\mathsf H}_\omega f)(x) = \int_{0}^{\infty} h_\omega(xy) \, f(y) \, dy 
\end{equation}
defines a bounded operator from $L^2((0,\infty),dx)$ to $L^2((0,\infty),dx)$ 
(Lemma \ref{lem_401}). 
For $a >0$, we denote by ${\mathsf P}_a$ the orthogonal projection from $L^2((0,\infty),dx)$ to $L^2((0,a),dx)$, 
and define 
\begin{equation}\label{209}
{\mathsf H}_{\omega,a} := {\mathsf P}_a {\mathsf H}_\omega {\mathsf P}_a: 
L^2((0,a),dx) ~\to~ L^2((0,a),dx).
\end{equation}
A study of ${\mathsf H}_{\omega}$ and ${\mathsf H}_{\omega,a}$ yields a canonical system as follows:
\begin{theorem} \label{thm_4}
Suppose that $\omega>1$. $($It implies automatically that $\Theta_\omega$ is inner in $\C^+$.$)$ 
Then the operator ${\mathsf H}_{\omega,a}$ is a Hilbert-Schmidt type self-adjoint operator with a continuous kernel for every $a>1$, 
and ${\mathsf H}_{\omega,a}=0$ for $0<a\leq 1$.  
Moreover $1\pm {\mathsf H}_{\omega,a}$ are invertible for every $a>0$. Define 
\begin{equation*}
m(a) := m_\omega(a) =
\frac{\det(1+{\mathsf H}_{\omega,a})}{\det(1-{\mathsf H}_{\omega,a})} 
\end{equation*} 
by using Fredholm determinants. 
Then $m(a)$ is real-valued continuous function on $(0,\infty)$, and the canonical system 
\begin{equation*}
- a \frac{\partial }{\partial a} \begin{bmatrix} X_a(z) \\ Y_a(z) \end{bmatrix}  
= z \, \begin{bmatrix} 0 & -1 \\ 1 & 0 \end{bmatrix} 
 \begin{bmatrix} m(a)^{-2} & 0 \\ 0 & m(a)^{2} \end{bmatrix} 
 \begin{bmatrix} X_a(z) \\ Y_a(z) \end{bmatrix} \quad (0< a <\infty)
\end{equation*}
has the explicit solution $(X_a,Y_a)=(A_a,B_a)$ given by \eqref{416} in Section $4$ such that 
\begin{enumerate}
\item $A_a(z)$ and $B_a(z)$ are real entire functions as a function of $z$ for every fixed $a>0$,
\item $A_a(-z)=A_a(z)$ and $B_a(-z)=-B_a(z)$ as a function of $z$ for every fixed $a>0$, 
\item $(A_1(z),B_1(z))=(A^\omega(z),B^\omega(z))$ and 
\[
\displaystyle{\lim_{a \to 1+}(A_a(z),B_a(z))=\lim_{a \to 1-}(A_a(z),B_a(z))=(A^\omega(z),B^\omega(z))}
\] 
hold uniformly on every compact subset in $\C$, 
where $A^\omega(z)$ and $B^\omega(z)$ are real entire functions defined in \eqref{AB}. 
\end{enumerate} 
Furthermore, the canonical system can be transformed into the pair of Sch\"odinger equations 
\begin{equation*}
\left( -a\frac{\partial}{\partial a}a\frac{\partial}{\partial a} + V^\pm(a) \right) \psi^\pm (a,z) 
= z^2 \psi^\pm(a,z)
\end{equation*}  
with the pair of potentials 
\begin{equation*}
V^{\pm}(a) = \left( \frac{1}{m(a)}a\frac{\partial }{\partial a} m(a) \right)^2 
\mp a \frac{\partial}{\partial a}\left( \frac{1}{m(a)}a\frac{\partial }{\partial a} m(a) \right)
\end{equation*}
by taking $\psi^+(a,z) = m^{-1}(a)A_a(z)$ and $\psi^-(a,z)=m(a)B_a(z)$. 
\end{theorem}
%
%
The assumption $\omega>1$ in Theorem \ref{thm_4} is required to obtain 
a continuity of the kernel $h_\omega(xy)$ in the proof in Section 4.3 and 4.4, 
since $h_\omega(x)$ has a singularity at $x=n \in \Z_{>0}$ for $0<\omega \leq 1$. 
However, we observed that singularities at $x = n \in {\Z}>0$ are in $L^2$ exactly for $\omega > 1/2$, 
and are in $L^1$ for all $\omega > 0$. 
This implies that that the function $h_\omega(x)$ on any interval $[x_0,x_1]$ with $0<x_0<x_1<\infty$ 
lies in the same function spaces, and it affects the behavior of associated ${\mathsf H}_{\omega,a}$. 
In fact, ${\mathsf H}_{\omega,a}$ is a Hilbert-Schmidt type self-adjoint operator 
such that $1\pm {\mathsf H}_{\omega,a}$ are invertible for every $a>0$ if $\omega >1/2$ 
(Lemma \ref{lem_402} and \ref{lem_404} below), and is a compact self-adjoint operator for all $\omega>0$. 
In addition, the type of singularities at $x = n \in {\Z}>0$ presumably affects the canonical system 
since it is given by determinants of $1\pm {\mathsf H}_{\omega,a}$ if $\omega>1$. 

On the other hand, as mentioned in Section 1.2 and 1.4, $\Theta_\omega(z)$ is an inner function in $\C^+$ for all $\omega \geq 1/2$ unconditionally, 
and for all $\omega>0$ under RH. 

Therefore, it is plausible that all results of Theorem \ref{thm_4} can be extended to $\omega >1/2$ unconditionally
without essential difficulties. 
Moreover, it is expected that Theorem \ref{thm_4} is generalized to $\omega>0$ 
if we assume RH for $\zeta(s)$. 
See Section 5 for further comments on the validity of Theorem \ref{thm_4}. 
\medskip

Finally, we emphasize that the limit behavior $\lim_{a \to +\infty}(A_a(z),B_a(z))$ is still open even if $\omega>1$. 
The expected result is $\lim_{a \to +\infty}(E^\omega(0),0)=(\xi(\frac{1}{2}+\omega),0)$ 
if we note that $E$ is normalized as $E(0)=1$ in Section 1.1. 
Provably, this limit behavior is related to the arithmetic properties of $\zeta(s)$ in more deep level, 
because we need information for all $\{c_\omega(n)\}_{n \geq 1}$ to understand it 
differ from the situation that we need only finitely many $c_\omega(n)$'s 
to understand ${\mathsf H}_{\omega,a}$ for a finite range of $a$. 
However, we do not touch this problem further in this paper.   

%
%
\section{Proof of Proposition \ref{prop_1} and Theorems \ref{thm_2}}
%
%
\subsection{Proof of Proposition \ref{prop_1}} 
%
%
For convenience, we use variable $s=1/2-iz$. 
Put $\gamma(s)=\frac{1}{2}s(s-1)\pi^{-s/2}\Gamma(s/2)$ so that $\xi(s)=\gamma(s)\zeta(s)$. Then
\begin{equation*}
\frac{\gamma(s-\omega)}{\gamma(s+\omega)}
= \pi^{\omega}
\frac{\Gamma\left(\frac{s-\omega}{2}+1\right)}{\Gamma\left(\frac{s+\omega}{2}+1 \right)} 
 -  \frac{2\omega\pi^{\omega}}{s+\omega-1}\frac{\Gamma\left(\frac{s-\omega}{2} +1\right)}{\Gamma\left(\frac{s+\omega}{2}+1 \right)}.
\end{equation*}
We have 
\begin{equation}\label{301}
\frac{\Gamma(\frac{s-\omega}{2}+1)}{\Gamma(\frac{s+\omega}{2}+1)} 
= \frac{2}{\Gamma(\omega)}\,\int_{0}^{1}  x^{2-\omega} (1-x^2)^{\omega-1}\, x^s \, \frac{dx}{x} 
\end{equation}
for $\Re(s+2)>\omega>0$ by \cite[(5.35) of p.195]{MR0352890}, and 
\begin{equation}\label{302}
\frac{1}{s+\omega-1} = \int_{0}^{1} x^{\omega-1}\, x^{s} \, \frac{dx}{x}
\end{equation}
for $\Re(s)>1-\omega$. Applying Theorem 44 of \cite{MR942661} to \eqref{301} and \eqref{302} 
together with
\begin{equation*}
\aligned
\frac{2}{\Gamma(\omega)}\,\int_{y}^{1}  x^{2-\omega} (1-x^2)^{\omega-1}\, (y/x)^{\omega-1} \, \frac{dx}{x} 
= \frac{y^{\omega-1} }{\Gamma(\omega)}\,\beta\left(y^2, \frac{3}{2} - \omega,  \omega \right), 
\endaligned
\end{equation*}
we obtain 
\begin{equation*}
\int_{0}^{\infty} g_\omega(x) \, x^{s} \, \frac{dx}{x}
=
\int_{0}^{1} g_\omega(x) \, x^{s} \, \frac{dx}{x}
=
\frac{\gamma(s-\omega)}{\gamma(s+\omega)} 
\end{equation*}
for $\Re(s)>{\rm max}(\omega-2,1-\omega)$. 
On the other hand, we have  
\begin{equation*}
\frac{\zeta(s-\omega)}{\zeta(s+\omega)} 
= \sum_{m=1}^{\infty} \frac{\mu(m)m^{-\omega}}{m^s} \sum_{n=1}^{\infty} \frac{n^{\omega}}{n^s} 
= \sum_{n=1}^{\infty} \frac{1}{n^s} \sum_{d|n} \frac{\mu(d)}{d^\omega} \left(\frac{n}{d}\right)^{\omega} = \sum_{n=1}^{\infty} \frac{c_\omega(n)}{n^s}
\end{equation*}
by definition \eqref{203}, where the series converges absolutely for $\Re(s)>1+\omega$. 
By definition \eqref{204}, we have formally 
\begin{equation*}
\aligned
\int_{0}^{\infty} h_\omega(x) \, x^{1-s} \, \frac{dx}{x} 
&= \sum_{n=1}^{\infty} c_\omega(n) \int_{0}^{\infty} x^{-1} g_\omega(n/x) \, x^{1-s} \, \frac{dx}{x} \\
&= \sum_{n=1}^{\infty} \frac{c_\omega(n)}{n^s} \int_{0}^{\infty} g_\omega(n/x) \, (n/x)^{s} \, \frac{dx}{x} 
= \frac{\gamma(s-\omega)}{\gamma(s+\omega)} \frac{\zeta(s-\omega)}{\zeta(s+\omega)},
\endaligned
\end{equation*}
and it is justified by Fubini's theorem for $\Re(s)>1+\omega$. Replacing $s$ by $1/2-iz$, we obtain \eqref{206}. 
%
%
\subsection{Proof of Theorem \ref{thm_2}} \hfill
%
%
It is sufficient to prove the following three assertions: \\
i) condition (1) is equivalent that $\Theta_\omega$ is inner in $\C^+$, 
ii) condition (2) implies that $\Theta_\omega$ is inner in $\C^+$, and 
iii) condition (3) implies that $\Theta_\omega$ is inner in $\C^+$, 
since  (1) implies (2) and (3) by definition of $\widehat{\Theta}_\omega$ and $h_\omega$. 
We prove them after the following lemma. 

\begin{lemma} \label{lem_301} 
Assume that $\Theta_\omega H^2 \subset H^2$. 
Then $\Theta_\omega$ is inner in $\C^+$. 
\end{lemma}
\begin{proof} 
Let $\delta>0$. We find that $\Theta_\omega(z)$ is uniformly bounded on the upper half-plane $\Im(z) \geq 1/2+\omega+\delta$ 
by using a usual estimate for the Dirichlet series $\zeta(s-\omega)/\zeta(s+\omega)$ 
and the Stirling formula for the gamma-function. 
On the other hand, we know \eqref{106}, and the assumption implies that $\Theta_\omega$ has no poles in $\C^+$.  
Hence, by applying the Phragm\'en-Lindel\"of convexity principle 
to $\Theta_\omega$ in the strip $0 \leq \Im(z) \leq 1/2+\omega+\delta$, 
we find that $\Theta_\omega$ is bounded on $0 \leq \Im(z) \leq 1/2+\omega+\delta$. 
Therefore $\Theta_\omega$ is a bounded analytic function is $\C^+$ satisfying \eqref{106}.  
This is the definition of an inner function in $\C^+$. 
\end{proof}

{\bf i) } 
Suppose that $\Theta_\omega$ is inner in $\C^+$. Then $\Theta_\omega F \in H^2$ for every $F \in H^2$. 
Thus the inverse (shifted) Fourier transform along the line $\Im(z)=c$
\begin{equation*}
\widehat{\Theta}_\omega f(x) = \frac{1}{2\pi}\int_{\Im(z)=c} \Theta_\omega(z)F(z) x^{-\frac{1}{2} - iz} \, dz
\end{equation*}
is independent of $c >0$, and belongs to $L^2((1,\infty),dx)$, 
where $f = {\mathsf F}_{1/2}^{-1}F$ and the integral converges in the sense of $L^2$. 
On the other hand 
\begin{equation*}
(h_\omega \ast f)(x) = \frac{1}{2\pi}\int_{\Im(z)=c'} \Theta_\omega(z)F(z) x^{-\frac{1}{2} - iz} \, dz
\end{equation*}
for $c'>1/2+\omega$ by Proposition \ref{prop_1} and \cite[Theorem 65]{MR942661}, 
where the integral converges also in the sense of $L^2$. 
Comparing these two formula for large $c$, we obtain (1). 

Conversely, suppose that (1) holds. Write $g = \widehat{\Theta}_\omega f = h_\omega \ast f$ 
for arbitrary fixed $f \in L^2((1,\infty),dx)$. 
Then $g$ belongs to $L^2((0,\infty),dx)$, since $\widehat{\Theta}_{\omega}$ maps $L^2((0,\infty),dx)$ to $L^2((0,\infty),dx)$ by definition. 
In addition, $g$ has a support in $[1,\infty)$, since both $h_{\omega}$ and $f$ have support in $[1,\infty)$. 
Therefore $g$ belongs to $L^2((1,\infty),dx)$. 
Because $f$ was arbitrary, 
we have $\Theta_\omega H^2 \subset H^2$. Hence $\Theta_\omega$ is inner in $\C^+$ by Lemma \ref{lem_301}. 
\hfill $\Box$
\medskip

{\bf ii)} 
Suppose that (2) holds. Then it implies $\widehat{\Theta}_\omega L^2((1,\infty),dx) \subset L^2((1,\infty),dx)$, 
since $\widehat{\Theta}_{\omega}$ maps $L^2((0,\infty),dx)$ to $L^2((0,\infty),dx)$ by its definition. 
It means $\Theta_\omega H^2 \subset H^2$ by definition of $\widehat{\Theta}_\omega$. 
Hence $\Theta_\omega$ is inner in $\C^+$ by Lemma \ref{301} \hfill $\Box$
\medskip

{\bf iii)} 
Suppose that (3) holds. Then $h_\omega \ast f$ belongs to $L^2((1,\infty),dx)$ for every $f$ in $L^2((1,\infty),dx)$, 
since $h_\omega \ast f$ has a support in $[1,\infty)$ by its definition. 
Therefore 
\begin{equation*}
\int_{0}^{\infty} h_\omega \ast f(x) \, x^{\frac{1}{2}+iz} \,  \frac{dx}{x} \in H^2.
\end{equation*}
%
Additionally, we suppose that $f$ belongs to the dense subset $L^1((1,\infty),dx) \cap L^2((1,\infty),dx)$. 
Then 
\begin{equation*}
\int_{0}^{\infty} h_\omega \ast f(x) \, x^{\frac{1}{2}+iz} \,  \frac{dx}{x} =\Theta_\omega(z)F(z)
\end{equation*}
for $\Im(z)>1/2+\omega$ by \cite[Theorem 44]{MR942661}. Therefore 
\begin{equation*}
\widehat{\Theta}_\omega \Bigl( L^1((1,\infty),dx) \cap L^2((1,\infty),dx) \Bigr) 
\subset L^2((1,\infty),dx),
\end{equation*}
This implies that $\widehat{\Theta}_\omega L^2((1,\infty),dx) \subset L^2((1,\infty),dx)$, 
since $\widehat{\Theta}_\omega$ is continuous by its definition.  
Therefore $\Theta_\omega H^2 \subset H^2$ by definition of $\widehat{\Theta}_\omega$, 
and hence $\Theta_\omega$ is inner in $\C^+$ by Lemma \ref{301}. \hfill $\Box$

%
%
\section{Proof of Theorem \ref{thm_4}}
%
%

In this section, we study operators \eqref{208}, \eqref{209}, and their kernels toward Theorem \ref{thm_4} 
referring to Burnol \cite{B1}. 
However here we use classical arguments rather than the theory of distributions used in  \cite{B1}. 

%
%
\subsection{Fredholm integral equations}
%
%
\begin{lemma} \label{lem_401}
Suppose that $\Theta_\omega$ is inner in $\C^+$. 
Define ${\mathsf H}_\omega f$ by integral \eqref{208} for compactly supported smooth functions $f$. 
Then ${\mathsf H}_\omega f$ belongs to $L^2((0,\infty),dx)$, 
and the linear map $f \mapsto {\mathsf H}_\omega f$ is extended to the isometry ${\mathsf H}_\omega: L^2((0,\infty),dx) \to L^2((0,\infty),dx)$ satisfying 
\begin{equation} \label{401}
({\mathsf F}_{1/2}{\mathsf H}_\omega f)(z) = {\Theta}_\omega(z)\,({\mathsf F}_{1/2} f)(-z)
\end{equation}
for $z \in \R$. Moreover, \eqref{401} holds for $\Im(z) \geq 0$, if $f \in L^2((0,\infty),dx)$ has a support in $[0,b]$ for some $b>0$.   
\begin{remark} This is applied unconditionally
to $\omega \geq 1/2$, and also to $0<\omega<1/2$ under RH by discussion in Section 1.2 and 1.4. 
\end{remark}
\end{lemma}
\begin{proof} If $f$ is a compactly supported smooth function, we have 
\begin{equation*}
\aligned
({\mathsf F}_{1/2}{\mathsf H}_\omega f)(z) 
& = \int_{0}^{\infty}  \int_{0}^{\infty} h_\omega(xy) \,  \, x^{\frac{1}{2}+iz} \, \frac{dx}{x} \, f(y) \, dy \\
& =  \int_{0}^{\infty} h_\omega(y) \,  \, x^{\frac{1}{2}+iz} \, \frac{dx}{x}  \int_{0}^{\infty} f(y) \, y^{\frac{1}{2}-iz}\, \frac{dy}{y} 
 = \Theta(z) F(-z) \quad (F={\mathsf F}_{1/2}f) \\
\endaligned
\end{equation*}
for $\Im(z)>1/2+\omega$ by Proposition \ref{prop_1}, 
and $F(-z)$ is an entire function satisfying 
$F(-z)=O(|z|^{-n})$ as $|z| \to \infty$ in any horizontal strip $c_1 \leq \Im(z) \leq c_2$ for arbitrary fixed $n>0$. 
Therefore, we find that ${\mathsf H}_\omega f$ belongs to $L^2((0,\infty),dx)$  
by applying the Fourier inversion formula to $\Theta_\omega(z)F(-z)$ along a line  $\Im(z)=c>1/2+\omega$ 
and then moving the path of integration to the real line $\Im(z)=0$, 
since $\Theta_\omega$ is inner in $\C^+$ by assumption. 
Moreover 
\begin{equation*}
\Vert {\mathsf H}_\omega f \Vert
= \Vert \Theta_\omega(\cdot) F(-\cdot) \Vert 
= \Vert F \Vert 
= \Vert f \Vert
\end{equation*} 
by \eqref{106}. 
Recall that the set of all compactly supported smooth function in $L^2((0,\infty),dx)$ is dense in $L^2((0,\infty),dx)$. 
Therefore $f \mapsto {\mathsf H}_{\omega}f$ is extended to all $f \in L^2((0,\infty),dx)$ by continuity, 
and the extended operator is obviously isometric. 

Equality \eqref{401} holds for real $z$ by the continuity. 
Suppose that $f\in L^2((0,\infty),dx)$ has a support in $[0,b]$ for some $b>0$. 
Then ${\mathsf H}_\omega f$ belongs to $L^2((0,\infty),dx)$ and has a support in $[1/b,\infty)$. 
Therefore the left-hand side of \eqref{401} is defined by the shifted Fourier integral and analytic in $\C^+$. 
On the other hand, $({\mathsf F}_{1/2} f)(-z)$ in the right-hand side of \eqref{401} 
is also defined by the shifted Fourier integral and analytic in $\C^+$, 
since $f$ has a support in $[0,b]$ by the assumption. 
Hence both sides of \eqref{401} are analytic functions in $\C^+$, 
and they are equal on the real line. 
Thus equality \eqref{401} holds for $\Im(z)\geq 0$.
\end{proof}
\begin{lemma} \label{lem_402}
Suppose that $\omega > 1/2$. $($It implies automatically that $\Theta_\omega$ is inner in $\C^+$.$)$ 
Then the operator ${\mathsf H}_{\omega,a}={\mathsf P}_a{\mathsf H}_{\omega}{\mathsf P}_a$ defined in \eqref{209} 
is a self-adjoint Hilbert-Schmidt type operator if $a>1$, 
and ${\mathsf H}_{\omega,a}=0$ if $0<a \leq 1$. 
\end{lemma}
\begin{proof} If $0<a \leq 1$ and $0<x<a$, we have  
\begin{equation*}
{\mathsf H}_\omega {\mathsf P}_a f(x) = \int_{0}^{a} h_{\omega}(xy)f(y)\,dy =0,  
\end{equation*}
since $h_{\omega}(x)=0$ for $0<x<1$, and $0 \leq xy < a^2 \leq 1$. 
Hence ${\mathsf H}_{\omega,a}=0$ if $0<a \leq 1$. 

Denote by $K(x,y)=h_{\omega}(xy)$ the kernel of ${\mathsf H}_{\omega,a}$. 
We have $K(x,y)=\overline{K(y,x)}$, since $h_\omega(xy)$ is real-valued. 
Thus ${\mathsf H}_{\omega,a}$ is self-adjoint. 
For $a>1$, we have 
\begin{equation*}
\aligned
\int_{0}^{a} \int_{0}^{a} |K(x,y)|^2 \, dx dy 
& = \int_{1/a}^{a} \int_{1/a}^{a} |h_\omega(xy)|^2 \, dx dy \\
& \leq \int_{1/a}^{a}  \frac{dy}{y} \, \int_{1/a^2}^{a^2} |h_\omega(x)|^2 \, dx 
= 2 \log a  \int_{1}^{a^2} |h_\omega(x)|^2 \, dx.   
\endaligned
\end{equation*}
Here $\int_{1}^{a^2} |h_\omega(x)|^2 \, dx < \infty$ 
if $\omega>1/2$, 
since $h_\omega(x)$ has only finitely many singularities 
at $x=n$ ($1 \leq n \leq \lfloor a^2 \rfloor$) in $[1,a^2]$, 
and $h_\omega(x) \ll |x-n|^{\omega-1}$ around $x=n$ by \eqref{201} and \eqref{204}. 
Hence, $K(x,y)=h_\omega(xy)$ is a Hilbert-Schmidt kernel if $a>1$ and $\omega>1/2$. 
\end{proof}
\begin{lemma} \label{lem_403}
Let $a>0$. Suppose that $\omega > 1/2$. 
Then the support of ${\mathsf H}_{\omega}{\mathsf P}_{a}f$ is not compact 
for $f \in L^2((0,\infty),dx)$ unless ${\mathsf H}_{\omega}{\mathsf P}_{a}f =0$.  
\end{lemma}
\begin{proof} We prove this by contradiction. 
Suppose that ${\mathsf H}_{\omega}{\mathsf P}_{a} f \not=0$ and has a compact support. 
Then ${\mathsf F}_{1/2}{\mathsf H}_{\omega}{\mathsf P}_{a} f$ is an entire function of exponential type by the Paley-Wiener theorem.  
On the other hand, we have
\begin{equation*}
{\mathsf F}_{1/2}{\mathsf H}_{\omega}{\mathsf P}_{a} f(z) 
= \Theta_{\omega}(z) \cdot {\mathsf F}_{1/2}{\mathsf P}_{a} f(-z).
\end{equation*}
This implies that $G(z):={\mathsf P}_{a} f(-z)/\xi(\frac{1}{2}+\omega-iz) $ is entire, 
because  \eqref{101} holds unconditionally for the denominator $E(z)=\xi(\frac{1}{2}+\omega-iz)$ 
of $\Theta_\omega$ defined in \eqref{105_1}, and $E(z)\not=0$ on $\Im(z)=0$. 
Thus, we have  
\begin{equation*}
{\mathsf F}_{1/2}{\mathsf H}_{\omega}{\mathsf P}_{a} f(z)  = \xi\left(\frac{1}{2}-\omega-iz\right) \cdot G(z),
\end{equation*}
where the right-hand side is a product of entire functions. 
The point is that the zeros in the numerator of $\Theta_\omega$ 
can not kill the poles of the denominator, which therefore must be killed by zeros of ${\mathsf P}_{a} f(-z)$. 
This allows $\xi(\frac{1}{2}-\omega-iz)$ to be factored out.

The entire function on the right-hand side has at least $\frac{1}{\pi}T \log T$ zeros 
in the disk of radius $T$ around the origin, as $T \to \infty$ (\cite[Theorem 9.4]{MR882550}). 
However all entire functions of exponential type have at most $O(T)$ zeros in the disk of radius $T$ around the origin, as $T \to \infty$, 
because of the Jensen formula (\cite[\S2.5 (15)]{Le}). This is a contradiction. 
(The proof contained an error in the first version, but it was revised by the reviewer.)
\end{proof}
\begin{lemma} \label{lem_404}
Let $\omega > 1/2$ and $a>1$. 
We have 
${\rm i)}$ ${\mathsf H}_{\omega,a} f =0$ for every $f \in L^2((0,1/a),dx)$,  
${\rm ii)}$ $\Vert {\mathsf H}_{\omega,a} f \Vert < \Vert f \Vert$ 
for every $0 \not=f \in L^2((0,a),dx)$, and 
${\rm iii)}$ $\Vert {\mathsf H}_{\omega,a} \Vert < 1$. 
In particular, $1 \pm {\mathsf H}_{\omega,a}$ are invertible operator on $L^2((0,a),dx)$.  
\end{lemma}
\begin{proof} 
If $0<x<1/a$, we  have   
\begin{equation*}
{\mathsf H}_\omega {\mathsf P}_a f(x) = \int_{0}^{a} h_\omega(xy)f(y)\,dy =0,  
\end{equation*}
since $h_\omega(x)=0$ for $0<x<1$, and $0 \leq xy < 1$. 
Hence i) is proved. 

To prove ii), 
it is sufficient to show $\Vert {\mathsf H}_{\omega,a} f \Vert \not= \Vert f \Vert$ unless $f=0$, 
because $\Vert {\mathsf H}_\omega \Vert =1$, 
$\Vert {\mathsf H}_{\omega,a} \Vert \leq \Vert {\mathsf P}_a \Vert \cdot \Vert {\mathsf H}_{\omega} \Vert \cdot \Vert {\mathsf P}_a \Vert = 1$, 
and 
$\Vert {\mathsf H}_{\omega,a} f \Vert \leq \Vert {\mathsf H}_{\omega,a} \Vert \cdot \Vert f \Vert \leq \Vert f \Vert$.  
Here $\Vert {\mathsf H}_{\omega,a} f \Vert \not= \Vert f \Vert$ is equivalent to $\Vert {\mathsf P}_a {\mathsf H}_\omega f \Vert \not= \Vert f \Vert$, 
since ${\mathsf P}_a f = f$ for $f \in L^2((0,a),dx)$. 
Suppose that $\Vert {\mathsf P}_a {\mathsf H}_\omega f \Vert = \Vert f \Vert$ for some $0 \not= f \in L^2((0,a),dx)$.  
Then it implies $\Vert {\mathsf P}_a {\mathsf H}_\omega f \Vert= \Vert {\mathsf H}_\omega f\Vert$ 
by $\Vert {\mathsf H}_\omega f \Vert = \Vert f \Vert$. 
Therefore
\begin{equation*}
\int_{0}^{a} |{\mathsf H}_\omega f(x)|^2 \, dx 
= \int_{0}^{\infty} |{\mathsf H}_\omega f(x)|^2 \, dx.
\end{equation*}
Thus ${\mathsf H}_\omega f(x) = 0$ for almost every $x>a$. On the other hand, we have 
\begin{equation*}
{\mathsf H}_\omega f(x) = \int_{0}^{a} h_\omega(xy)f(y)\,dy 
= \int_{1/x}^{a} h_\omega(xy)f(y)\,dy =0 
\end{equation*}
for $0<x<1/a$ by $f\in L^2((0,a),dx)$. Hence ${\mathsf H}_\omega f$ has a compact support contained in $[1/a,a]$. 
However, it is impossible for any $f\not=0$ by Lemma \ref{lem_403}. 
As the consequence $\Vert {\mathsf H}_{\omega,a} f \Vert < \Vert f \Vert$ for $0\not=f \in L^2((1/a,a),dx)$. 

Finally, we prove iii). By Lemma \ref{lem_402}, ${\mathsf H}_{\omega,a}$ is a self-adjoint compact operator. 
Therefore, ${\mathsf H}_{\omega,a}$ has purely discrete spectrum which has no accumulation points except for $0$, 
and one of $\pm \Vert {\mathsf H}_{\omega,a} \Vert$ is an eigenvalue of ${\mathsf H}_{\omega,a}$. 
However, by ii), every eigenvalue of ${\mathsf H}_{\omega,a}$ has an absolute value less than $1$. 
Hence $\Vert {\mathsf H}_{\omega,a} \Vert <1$. 
\end{proof}

\begin{lemma} \label{lem_405}
Let $\omega >1/2$, $a>1$ and $\varepsilon \in \{\pm 1\}$. Then the integral equation
\begin{equation} \label{402}
X(x) + \varepsilon \int_{0}^{a} h_\omega(xy) X(y) \, dy 
= h_\omega(ax)
\end{equation}
has unique solution $X= \phi_{a}^{\,\varepsilon}$ in $L^2((0,a),dx)$, 
which is real-valued almost everywhere in $[0,a]$ and vanishes almost everywhere in $[0,1/a]$.  

Moreover, if $\omega>1$, the solution $\phi_{a}^{\varepsilon}$ is a real-valued continuous function on $[0,a]$ vanishing on $[0,1/a]$. 
\end{lemma}
\begin{proof}
By Lemmas \ref{lem_402} and \ref{lem_404}, ${\mathsf H}_{\omega,a}$ is a compact operator such that $\pm 1$ belong to its resolvent set. 
Therefore, integral equation \eqref{402} has unique solution $\phi_{a}^{\varepsilon}$ in $L^2((0,a),dx)$ by the Fredholm alternative. 
We have $h_\omega(ax)=0$ and $\int_{0}^{a} h_\omega(xy) \phi_{a}^{\varepsilon}(y) \, dy=0$ for almost every $0<x<1/a$, 
since $0<xy<1$ for $0<y\leq a$, and $h_\omega(x)=0$ for $0<x<1$.  

On the other hand, if $\omega>1$, the integral $\int_{0}^{a} h_\omega(xy) f(y) \, dy$ 
defines a continuous function on $[0,a]$ which vanishes on $[0,1/a]$ for every $f \in L^2((0,a),dx)$, 
since the kernel $h_\omega(xy)$ is continuous on $[0,a] \times [0,a]$ by \eqref{201}. 
Hence $\phi_{a}^{\varepsilon}$ is continuous on $[0,a]$ and $\phi_{a}^{\varepsilon}(x)=0$ for $0\leq x \leq 1/a$. 
Obviously $\phi_a^\varepsilon$ is real-valued, since $h_\omega(x)$ is real-valued, 
\end{proof}

\begin{lemma} \label{lem_406}
Let $\omega>1/2$, $a>1$ and $\varepsilon \in \{\pm 1\}$. 
Then the integral equation \eqref{402} 
has unique extended solution $X= \tilde{\phi}_{a}^{\,\varepsilon}$ in $L^2((0,b),dx)$ for arbitrary $b>a$, 
which is real-valued almost everywhere in $[0,b]$, 
and $\tilde{\phi}_{a}^{\,\varepsilon}(x) = \phi_{a}^{\,\varepsilon}(x)$ for almost every $0<x<a$. 

Moreover, if $\omega>1$, the integral equation \eqref{402} 
has unique extended solution $X= \tilde{\phi}_{a}^{\,\varepsilon}$ in $C^0(0,\infty)$, 
which is real-valued on $[0,\infty)$ and   satisfies $\tilde{\phi}_{a}^{\,\varepsilon}(x) = \phi_{a}^{\,\varepsilon}(x)$ for $0<x<a$. 
\end{lemma}
\begin{proof}
The solution $\phi_{a}^{\,\varepsilon}$ of Lemma \ref{lem_405} is extended to 
the solution $\tilde{\phi}_{a}^{\,\varepsilon}$ on $(0,b)$ by 
\begin{equation} \label{404_1}
\tilde{\phi}_{a}^{\,\varepsilon}(x) 
= h_\omega(ax) -  \varepsilon \int_{0}^{a} h_\omega(xy) \phi_{a}^{\,\varepsilon}(y) \, dy. 
\end{equation}
The right-hand side belongs to $L^2((0,b),dx)$ by the Cauchy-Schwartz inequality, 
since $h_\omega(x)$ belongs to $L^2((0,b'),dx)$ for every $0<b'<\infty$ when $\omega>1/2$ 
and the integral on the right-hand side vanishes for almost every $0<x<1/a$. 
Clearly, $\tilde{\phi}_{a}^{\,\varepsilon}(x) = \phi_{a}^{\,\varepsilon}(x)$ for almost every $0<x<a$.  
Conversely, equality \eqref{404_1} shows that every solution of \eqref{402} on $(0,b)$ is determined by its restriction on $(0,a)$. 
Hence the uniqueness of solutions follows from Lemma \ref{lem_405}. 
By the way of the extension, $\tilde{\phi}_{a}^{\,\varepsilon}$ is real-valued almost everywhere. 

If $\omega>1$, we obtain unique extended continuous solution $\tilde{\phi}_{a}^{\,\varepsilon}$ on $(0,\infty)$ by \eqref{404_1}, 
since $h_\omega(x)$ is continuous on $(0,\infty)$ and $C^0(0,a) \subset L^2((0,a),dx)$. 
\end{proof}

In what follows, we denote by $\phi_{a}^{\,\varepsilon}$ the extended solution $\tilde{\phi}_{a}^{\,\varepsilon}$ for $a>1$ 
if no confusion arise. 
For $0<a \leq 1$, we take the convention that 
\[
\phi_a^+(x) = \phi_a^-(x) = h_\omega(ax) \quad x \in (0,\infty). 
\]
Obviously, these are continuous on $(0,\infty)$ if $\omega>1$.  
This convention is compatible with Lemma \ref{lem_405} and \ref{lem_406}, 
since integral equation \eqref{402} for $0<a \leq 1$ should be $X(x) = h_\omega(ax)$ by Lemma \ref{lem_402}, 
and $h_\omega(ax)=0$ on $(0,a)$ for $0<a\leq 1$. 
Then its extension $\tilde{\phi}_{a}^{\,\varepsilon}(x)$ to $(0,\infty)$ should be $h_\omega(ax)$ by \eqref{404_1}. 
%
%
\subsection{Differentiability of the solution}
%
%
In this part, we handle the differentiability of the extended solution $\phi_a^\varepsilon(x)$ with respect to $x$ and $a$ 
under the restriction to the parameter $\omega>1$. 
This restriction is required in order to obtain the continuity of the kernel $K(x,y)=h_\omega(xy)$. 
\medskip

Let $a>1$. 
The solution $\phi_{a}^{\,\varepsilon}$ of \eqref{402} 
is related to the kernel of the resolvent $(1 - \lambda {\mathsf H}_{\omega,a})^{-1}$ as follows. 
The kernel $K(x,y)=h_\omega(xy)$ of ${\mathsf H}_{\omega,a}$ is continuous on $[0,a] \times [0,a]$ by the assumption $\omega>1$. 
Then there exists a continuous function $R(x,y;\lambda;a)$ for $(x,y,\lambda) \in [0,a] \times [0,a] \times \C$ satisfying integral equations
\begin{equation}\label{403}
\aligned
R(x,y;\lambda;a) - \lambda \int_{0}^{a} K(x,z) R(z,y;\lambda;a) \, dz  
&=  K(x,y), \\ 
R(x,y;\lambda;a) - \lambda \int_{0}^{a} K(z,y) R(x,z;\lambda;a) \, dz 
&= K(x,y)
\endaligned
\end{equation}
(see Smithies~\cite[Chap. V]{MR0104991}, Lax~\cite[Chap. 24]{MR1892228}, for example). 
%
By taking $y=a$ and $\lambda=-\varepsilon$ in the first equation of \eqref{403}, we have 
\begin{equation*}
R(x,a;-\varepsilon;a) + {\varepsilon}\int_{0}^{a} h_\omega(xz) R(z,a;-\varepsilon;a) \, dz =  h_\omega(ax).
\end{equation*}
Therefore, we obtain 
\begin{equation} \label{Rker_1}
\phi_a^\varepsilon(x) = R(x,a;-\varepsilon;a)
\end{equation}
for $0< x < a$ by the uniqueness of solutions of \eqref{402}. 
In particular, we obtain the continuity of $\phi_a^\varepsilon(x)$ for $x$ again, and  
\begin{equation} \label{404}
\lim_{a \to 1^+} \phi_a^\varepsilon(a) = \lim_{a \to 1^+}R(a,a;-\varepsilon;a)=0
\end{equation}
by Lemma \ref{lem_402}. 
We investigate the differentiability of $\phi_a^\varepsilon(x)$ by using the resolvent kernel $R(x,y;\lambda;a)$. 
The following inequality is going to be used often. 
\medskip

\noindent
{\bf Hadamard's inequality} (see \cite[Theorem 5.2.1]{MR0104991}, for example).  
Let $A=(a_{ij})$ be a $n \times n$ complex matrix. If $|a_{ij}| \leq M$ $(1 \leq i,j \leq n)$, then 
$|\det A|^2 \leq n^n M^{2n}$. 
\medskip

We introduce the notation 
\[
K\begin{pmatrix} x_1,x_2,\cdots, x_n \\ y_1,y_2,\cdots, y_n \end{pmatrix}
= \det 
\begin{pmatrix} 
K(x_1,y_1) & K(x_1,y_2) & \cdots & K(x_1,y_n) \\
K(x_2,y_1) & K(x_2,y_2) & \cdots & K(x_2,y_n) \\
\vdots & \vdots & \ddots & \vdots \\ 
K(x_n,y_1) & K(x_n,y_2) & \cdots & K(x_n,y_n)
\end{pmatrix}
\]
as usual. 
The Fredholm determinant $d(\lambda;a)$ and the first Fredholm minor $D(x,y;\lambda;a)$ 
of the continuous kernel $K(x,y)$ on $\Omega_a=[0,a]\times[0,a]$ are defined as follows: 
\begin{align} \label{Fd_1}
d(\lambda;a) & = \sum_{n=0}^{\infty} d_n(a) \lambda^n,  \\ \label{FD_1}
D(x,y;\lambda;a) & = \sum_{n=0}^{\infty} D_n(x,y;a) \lambda^n, 
\end{align}
where $d_0(a)=1$, $D_0(x,y;a)=K(x,y)$ and 
\begin{equation} \label{Fd_2}
d_n(a) = \frac{(-1)^n}{n!}\int_{0}^{a} \cdots \int_{0}^{a} K\begin{pmatrix} x_1,x_2,\cdots, x_n \\ x_1,x_2,\cdots, x_n \end{pmatrix} dx_1 \dots dx_n
\quad (n\geq 1),
\end{equation}
\begin{equation} \label{FD_2}
D_n(x,y;a) = \frac{(-1)^n}{n!}\int_{0}^{a} \cdots \int_{0}^{a} K\begin{pmatrix} x,x_1,\cdots, x_n \\ y,x_1,\cdots, x_n \end{pmatrix} dx_1 \dots dx_n
\quad (n\geq 1). 
\end{equation}
The kernel $D_n(x,y;a)$ are clearly continuous in $(x,y)$. 
It is well-known that the series \eqref{FD_1} 
converges uniformly and absolutely in $(x,y,\lambda)$ 
when $\lambda$ is confined in a compact subset of $\C$, 
and $D(x,y;\lambda;a)$ is a continuous function on $\Omega_a$ for every $\lambda \in \C$ 
(see \cite[Theorem 5.3.1]{MR0104991}, for example).
If $d(\lambda;a)\not=0$, the resolvent kernel $R(x,y;\lambda;a)$ is given by 
\begin{equation} \label{Rker_2}
R(x,y;\lambda;a) = \frac{D(x,y;\lambda;a)}{d(\lambda;a)}. 
\end{equation}
Note that $d(\pm 1;a)\not=0$ for every $a>1$ when $K(x,y)=h_\omega(xy)$ and $\omega>1/2$ 
by Lemma \ref{lem_404} and Theorem 5.6.1 of  \cite{MR0104991}. 

\begin{lemma} \label{lem_ad1}
Let $\omega >1$, $a>1$ and $\varepsilon \in \{\pm 1\}$. 
Then the extended solution $\phi_a^\varepsilon(x)$ is continuously differentiable 
on $x \in [0,\infty) \setminus \{n/a\,|\,n\in \N\}$. 
\end{lemma}
\begin{proof}
By \eqref{404_1}, \eqref{Rker_1} and \eqref{Rker_2}, 
it is sufficient to prove that 
the Fredholm minor $D(x,y;\lambda;a)$ for the kernel $K(x,y)=h_\omega(xy)$ on $\Omega_a=[0,a]\times[0,a]$ 
is continuously differentiable on $x \in D_a:=[0,a] \setminus \{n/a\,|\,n\in \N\}$ for every fixed $y \in [0,a]$, 
since $h_\omega(ax)$ is continuously differentiable on $ x \in [0,\infty) \setminus \{n/a\,|\,n\in \N\}$, 
and $|\frac{\partial}{\partial x}h_\omega(xy)\phi_a^\varepsilon(y)|$ is integrable on $[0,a]$. 
To prove it, we modify the proof of Theorem 5.3.1 in \cite{MR0104991}. We have 
\[
\aligned
K&\begin{pmatrix} x,x_1,\cdots, x_n \\ y,x_1,\cdots, x_n \end{pmatrix}
= \det 
\begin{pmatrix} 
K(x,y) & K(x,x_1) & \cdots & K(x,x_n) \\
K(x_1,y) & K(x_1,x_1) & \cdots & K(x_1,x_n) \\
\vdots & \vdots & \ddots & \vdots \\ 
K(x_n,y) & K(x_n,x_1) & \cdots & K(x_n,x_n)
\end{pmatrix} \\
& = K(x,y) K\begin{pmatrix} x_1,\cdots, x_n \\ x_1,\cdots, x_n \end{pmatrix}
+ 
\det 
\begin{pmatrix} 
0 & K(x,x_1) & \cdots & K(x,x_n) \\
K(x_1,y) & K(x_1,x_1) & \cdots & K(x_1,x_n) \\
\vdots & \vdots & \ddots & \vdots \\ 
K(x_n,y) & K(x_n,x_1) & \cdots & K(x_n,x_n)
\end{pmatrix}.
\endaligned
\]
Therefore, by \eqref{Fd_2} and \eqref{FD_2}, we have  
\[
\aligned
\,&\frac{\partial}{\partial x} D_{n}(x,y;a)  = d_n(a) \frac{\partial}{\partial x}K(x,y) \\
&+
\frac{(-1)^n}{n!}\int_{0}^{a}\cdots\int_{0}^{a} 
\det 
\begin{pmatrix} 
0 & \frac{\partial}{\partial x}K(x,x_1) & \cdots & \frac{\partial}{\partial x}K(x,x_n) \\
K(x_1,y) & K(x_1,x_1) & \cdots & K(x_1,x_n) \\
\vdots & \vdots & \ddots & \vdots \\ 
K(x_n,y) & K(x_n,x_1) & \cdots & K(x_n,x_n)
\end{pmatrix} dx_1 \dots dx_n \\
& = d_n(a) \frac{\partial}{\partial x}K(x,y) + D_n^\dagger(x,y;a),
\endaligned
\]
say. Then, we obtain 
\[
\frac{\partial}{\partial x}D(x,y;\lambda;a) = d(\lambda;a)\frac{\partial}{\partial x}K(x,y) 
+ \sum_{n=0}^{\infty} D_n^\dagger(x,y;a) \lambda^n 
\] 
by \eqref{FD_1}. The first term on the right-hand side is continuous on $x \in D_a$. 
Therefore, in order to prove the existence and the continuity of $\frac{\partial}{\partial x}D(x,y;\lambda;a)$ on $x \in D_a$, 
it is sufficient to prove that the series on the right-hand side 
converges uniformly on every compact subset in $D_a$. 
Put 
\[
M_1(a) = a \sup_{(x,y) \in \Omega_a} |K(x,y)|, \qquad M_2(a) =  \sup_{x \in [0,a]} \int_{0}^{a} \left|\frac{\partial}{\partial x}K(x,y)\right|\,dy. 
\]
The second constant $M_2(a)$ is well-defined, since 
$\int_{0}^{a} \left|\frac{\partial}{\partial x}K(x,y)\right| dy = \int_{0}^{a} |h_\omega^\prime(xy)|\,ydy$
is continuous on $[0,a]$ (by $\omega>1$).  
Using the row expansion of the determinant and Hadamard's inequality, we have 
\[
\left| 
\det 
\begin{pmatrix} 
0 & \frac{\partial}{\partial x}K(x,x_1) & \cdots & \frac{\partial}{\partial x}K(x,x_n) \\
K(x_1,y) & K(x_1,x_1) & \cdots & K(x_1,x_n) \\
\vdots & \vdots & \ddots & \vdots \\ 
K(x_n,y) & K(x_n,x_1) & \cdots & K(x_n,x_n)
\end{pmatrix} \right|
\leq n^{\frac{1}{2}n} \left(\frac{M_1(a)}{a}\right)^n \sum_{j=1}^{n}\left| \frac{\partial}{\partial x}K(x,x_j) \right|.
\]
Therefore, we obtain
\[
\aligned
| D_n^\dagger(x,y;a) | 
& \leq 
\frac{n^{\frac{1}{2}n}}{n!}  \left(\frac{M_1(a)}{a}\right)^n\int_{0}^{a}\cdots\int_{0}^{a} 
\sum_{j=1}^{n}\left| \frac{\partial}{\partial x}K(x,x_j) \right|dx_1 \dots dx_n \\
& = \frac{n^{\frac{1}{2}n}}{a n!} M_1(a)^n 
\sum_{j=1}^{n} \int_{0}^{a} \left| \frac{\partial}{\partial x}K(x,x_j) \right|dx_j
 \leq \frac{M_2(a)}{a}\frac{n^{\frac{1}{2}n}M_1(a)^n}{(n-1)!}.
\endaligned
\]
Therefore, the series $\sum_{n=0}^{\infty} D_n^\dagger(x,y;a) \lambda^n$ converges uniformly and absolutely  
in $(x,y,\lambda) \in \Omega_a \times \C$, 
when $\lambda$ is contained in a compact subset of $\C$.  
Hence, for fixed $y \in [0,a]$, $D(x,y;\lambda;a)-d(\lambda;a)K(x,y)$
is a continuously differentiable function on $[0,a]$ such that 
\[
\frac{\partial}{\partial x}\Bigl(D(x,y;\lambda;a)-d(\lambda;a)K(x,y) \Bigr) = \sum_{n=0}^{\infty} D_n^\dagger(x,y)\lambda^n.
\]
We complete the proof of the lemma. 
\end{proof}

\begin{lemma} \label{lem_ad2}
Let $\omega >1$ and $\varepsilon \in \{\pm 1\}$. 
Then the extended solution $\phi_a^\varepsilon(x)$ is continuous in $a \in (1,\infty)$ for every fixed $x>0$. 
In addition, it is continuously differentiable with respect to $a$ 
in $(1,\infty) \setminus \{n/x,\sqrt{n}\,|\,n \in \N\}$. 
\end{lemma}
\begin{proof} The continuity in $a$ follows from \eqref{404_1} and \eqref{Rker_1}. 
Before the proof of the differentiability, we note that 
$d(\lambda,a)$ is continuous in $a$. In fact, we have 
\[
|d_n(a)| \leq \frac{n^{\frac{1}{2}n}M_1(a)^n}{n!} \quad \left( M_1(a) = a \sup_{(x,y) \in \Omega_a} |K(x,y)| \right)
\]
by definition \eqref{Fd_2} and Hadamard's inequality, 
and hence the series of \eqref{Fd_1}
converges absolutely and uniformly on a compact subset of $(\lambda,a) \in \C \times (1,\infty)$. 

Let $\lambda \in \C$ such that $d(\lambda;a)\not=0$ for every $a>1$. We have  
\[
\frac{\partial}{\partial a}R(x,a;\lambda;a)
=\frac{\frac{\partial}{\partial a}D(x,a;\lambda;a)d(\lambda;a)-D(x,a;\lambda;a)\frac{\partial}{\partial a}d(\lambda;a)}{d(\lambda;a)^2}
\]
by \eqref{Rker_2}. 
Therefore, in order to prove the lemma, we need (i) the existence and the continuity of $\frac{\partial}{\partial a}d(\lambda;a)$ 
and (ii) the existence, the continuity and the integrability of $\frac{\partial}{\partial a}D(x,a;\lambda;a)$ 
by \eqref{404_1}, \eqref{Rker_1} and 
\[
\frac{\partial}{\partial a}\int_{0}^{a} h_\omega(xy)\phi_a^\varepsilon(y)\,dy
= h_\omega(ax)\phi_a^\varepsilon(a)
+\int_{0}^{a} h_\omega(xy) \frac{\partial}{\partial a}\phi_a^\varepsilon(y)\,dy.
\] 
We prove (i). By definition \eqref{Fd_1} and \eqref{Fd_2}, we have 
\[
\aligned
\frac{\partial}{\partial a} d(\lambda;a) 
& = - \lambda K(a,a) 
 + \sum_{n=2}^{\infty}
\frac{(-\lambda)^n}{n!}
\left\{
\int_{0}^{a} \cdots \int_{0}^{a} K\begin{pmatrix} a,x_2,\cdots, x_n \\ a,x_2,\cdots, x_n \end{pmatrix} dx_2 \dots dx_n \right. \\
& + \sum_{k=2}^{n-1}
\int_{0}^{a} \cdots \int_{0}^{a} K\begin{pmatrix} x_1,\cdots,x_{k-1},a,x_{k+1}, \cdots, x_n \\ x_1,\cdots,x_{k-1},a,x_{k+1}, \cdots, x_n \end{pmatrix} dx_1 \dots dx_{k-1}dx_{k+1} \dots dx_n \\
& \left. 
+ \int_{0}^{a} \cdots \int_{0}^{a} K\begin{pmatrix} x_1,x_2,\cdots, x_{n-1}, a \\ x_1,x_2,\cdots, x_{n-1}, a \end{pmatrix} dx_1 \dots dx_{n-1}
\right\}.
\endaligned
\]  
Clearly, each term in the series is continuous in $a$, since $K(x,y)$ is continuous. 
By using Hadamard's inequality, 
\[
\aligned
\left| \frac{\partial}{\partial a} d(\lambda;a) \right| 
& \leq \sum_{n=1}^{\infty} 
\frac{|\lambda|^n}{n!} n^{\frac{1}{2}n} \left(\frac{M_1(a)}{a}\right)^n n a^{n-1}
= \frac{1}{a}\sum_{n=1}^{\infty} 
\frac{n^{\frac{1}{2}n}}{(n-1)!}  (|\lambda|M_1(a))^n. 
\endaligned
\]
The series on the right-hand side converges uniformly 
on a compact subset of $(\lambda,a) \in \C \times [0,\infty)$.  
Hence $d(\lambda;a)$ is continuously differentiable for $a$. 

Successively, we prove (ii).  we have 
\[
\frac{\partial}{\partial a}D(x,a;\lambda;a)
= D_y(x,a;\lambda;a) + D_a(x,a;\lambda;a),
\]
where $D_y$ (resp. $D_a$) means the partial derivative with respect to the second (resp. the fourth) variable. 
We find that $D(x,y;\lambda;a)$ is continuously differentiable with respect to $y \in [0,\infty)\setminus\{n/a\,|\,n \in \N\}$, 
and $D_y(x,y;\lambda;a)$ is a continuous function on $(x,y) \in [0,\infty) \times ([0,\infty)\setminus\{n/a\,|\,n \in \N\})$  
by a way similar to the proof of Lemma \ref{lem_ad1}. 
Thus $D_y(x,a;\lambda;a)$ is continuous on $a \in (1,\infty) \setminus \{\sqrt{n}\,|\,n \in \N\}$ for fixed $x$, 
and $|D_y(x,a;\lambda;a)|$ is integrable on $[0,a]$ with respect to $x$. 
On the other hand, by definition \eqref{FD_1} and \eqref{FD_2},  
\[
\aligned
\frac{\partial}{\partial a}&D(x,y;\lambda;a) 
 = - \lambda K(a,a)K(x,y) + \lambda K(x,a)K(a,y) \\
& + \sum_{n=2}^{\infty}
\frac{(-\lambda)^n}{n!}
\left\{
\int_{0}^{a} \cdots \int_{0}^{a} K\begin{pmatrix} x,a,x_2,\cdots, x_n \\ y,a,x_2,\cdots, x_n \end{pmatrix} dx_2 \dots dx_n \right. \\
& + \sum_{k=2}^{n-1}
\int_{0}^{a} \cdots \int_{0}^{a} K\begin{pmatrix} x,x_1,\cdots,x_{k-1},a,x_{k+1}, \cdots, x_n \\ y,x_1,\cdots,x_{k-1},a,x_{k+1}, \cdots, x_n \end{pmatrix} dx_1 \dots dx_{k-1}dx_{k+1} \dots dx_n \\
& \left. + \int_{0}^{a} \cdots \int_{0}^{a} K\begin{pmatrix} x,x_1,\cdots, x_{n-1}, a \\ y,x_1,\cdots, x_{n-1}, a \end{pmatrix} dx_1 \dots dx_{n-1}
\right\}.
\endaligned
\]  
Clearly, each term in the series is continuous in $(x,y,a)$, since $K(x,y)$ is continuous. 
By the row expansion of the determinant and Hadamard's inequality, we obtain  
\[
\aligned
\left|\frac{\partial}{\partial a}D(x,y;\lambda;a)  \right| 
& \leq \sum_{n=1}^{\infty}
\frac{|\lambda|^n}{n!} n^{\frac{1}{2}n} \left(\frac{M_1(a)}{a}\right)^n\int_{0}^{a}\cdots\int_{0}^{a} 
\sum_{j=1}^{n}\left| K(x,x_j) \right|dx_1 \dots dx_{n-1} \\
& = \frac{1}{a^2}\int_{0}^{a} \left|K(x,x_1) \right|dx_1 \sum_{n=1}^{\infty}\frac{n^{\frac{1}{2}n}}{(n-1)!} (|\lambda|M_1(a))^n .
\endaligned
\]
when $0 \leq y \leq a$. The series on the right-hand side converges uniformly 
on a compact subset of $(\lambda,a) \in \C \times (1,\infty)$. 
Thus $D_a(x,a;\lambda;a)$ is continuous in $a$. 
In addition, the right-hand side shows that $|D_a(x,a;\lambda;a)|$ is integrable on $[0,a]$ with respect to $x$. 

Hence $\frac{\partial}{\partial a}R(x,a;\lambda;a)$ is continuous on $a \in (1,\infty) \setminus \{\sqrt{n}\,|\,n \in \N\}$ 
for fixed $x$, and $|\frac{\partial}{\partial a}R(x,a;\lambda;a)|$ is integrable on $[0,a]$ with respect to $x$. 
As a consequence we obtain the lemma by \eqref{404_1} and \eqref{Rker_1}. 
\end{proof}
%
%
\subsection{The first order differential system}
%
%
As in the previous section, we assume that $\omega>1$. 
Then $\Theta_\omega$ is an inner function in $\C^+$, the kernel $h_\omega(xy)$ of 
${\mathsf H}_{\omega}$ or ${\mathsf H}_{\omega,a}$ is continuous, 
and $\phi_a^{\,\varepsilon}(x)$ is continuously differentiable 
with respect to $x$ and $a$ outside loci $ax = k$ ($k \in \N$). 
Under this situation, we derive a first order differential system arising from $\phi_a^{\,\varepsilon}$ ($a>1,\,\varepsilon\in\{\pm 1\}$) 
start from 
\begin{equation} \label{405}
\phi_a^{\,\varepsilon}(x) + \varepsilon \int_{0}^{a} h_\omega(xy) \phi_a^{\,\varepsilon}(y) \, dy = h_\omega(ax). 
\end{equation}

Firstly, we operate $a\frac{\partial }{\partial a}$ on both side of \eqref{405}. Then,  
\begin{equation*}
a \frac{\partial}{\partial a} \phi_a^{\,\varepsilon}(x) 
+ \varepsilon a \frac{\partial}{\partial a}\int_{0}^{a} h_\omega(xy) \phi_a^{\,\varepsilon}(y) \, dy 
= a \frac{\partial}{\partial a} h_\omega(ax); 
\end{equation*}
\begin{equation*}
a \frac{\partial}{\partial a} \phi_a^{\,\varepsilon}(x) 
+ \varepsilon a \phi_a^{\,\varepsilon}(a)h_\omega(ax) + \varepsilon \int_{0}^{a} h_\omega(xy) 
a \frac{\partial}{\partial a} \phi_a^{\,\varepsilon}(y) \, dy 
= a \frac{\partial}{\partial a} h_\omega(ax);
\end{equation*}
\begin{equation} \label{406}
\aligned
a\frac{\partial}{\partial a} \phi_a^{\,\varepsilon}(x) 
+ \varepsilon \int_{0}^{a} h_\omega(xy)  a \frac{\partial}{\partial a} \phi_a^{\,\varepsilon}(y) \, dy 
&= - \varepsilon\, a \, \phi_a^{\,\varepsilon}(a) h_\omega(ax)+ a \frac{\partial}{\partial a} h_\omega(ax).
\endaligned
\end{equation}

Secondly, we operate $x\frac{\partial }{\partial x}$ on both side of \eqref{405}:  
\begin{equation*}
x\frac{\partial}{\partial x}\phi_a^{\,\varepsilon}(x) 
+ \varepsilon x \frac{\partial}{\partial x}\int_{0}^{a} h_\omega(xy) \phi_a^{\,\varepsilon}(y) \, dy 
= x\frac{\partial}{\partial x}h_\omega(ax) = a \frac{\partial}{\partial a} h_\omega(ax).
\end{equation*}
Using the identity $x\frac{\partial }{\partial x}h_\omega(xy) = y\frac{\partial }{\partial y}h_\omega(xy)$ 
and then applying integration by parts to the integral of the left-hand side, we have 
\begin{equation*}
x\frac{\partial}{\partial x}\phi_a^{\,\varepsilon}(x) 
+ \varepsilon\, a \, \phi_a^{\,\varepsilon}(a) h_\omega(ax) 
- \varepsilon \int_{0}^{a} h_\omega(xy) \frac{\partial}{\partial y} \Bigl( y \phi_a^{\,\varepsilon}(y) \Bigr) \, dy 
= a \frac{\partial}{\partial a} h_\omega(ax). 
\end{equation*}
Putting $\delta_x=x \frac{\partial }{\partial x} + \frac{1}{2}=\frac{\partial}{\partial x} x - \frac{1}{2}$, we obtain
\begin{equation} \label{407}
\aligned
\delta_x\phi_a^{\,\varepsilon}(x) - \frac{1}{2}\phi_a^{\,\varepsilon}(x) - \varepsilon \int_{0}^{a} h_\omega(xy) & 
\Bigl( \delta_y  \phi_a^{\,\varepsilon}(y) + \frac{1}{2} \phi_a^{\,\varepsilon}(y)\Bigr) \, dy \\
& = - \varepsilon \, a \, \phi_a^{\,\varepsilon}(a) h_\omega(ax) +  a \frac{\partial}{\partial a} h_\omega(ax).
\endaligned
\end{equation}
Next, we rewrite the left-hand side of \eqref{407} as follows 
by using \eqref{405} for the second term of the left-hand side:   
\begin{equation} \label{408}
\aligned
\delta_x \phi_a^{\,\varepsilon}(x) - \frac{1}{2}\Bigl( h_\omega(ax) 
- \varepsilon \int_{0}^{a} & h_\omega(xy) \phi_a^{\,\varepsilon}(y) \, dy \Bigr) 
 - \varepsilon \int_{0}^{a} h_\omega(xy) \Bigl( \delta_y  \phi_a^{\,\varepsilon}(y) + \frac{1}{2} \phi_a^{\,\varepsilon}(y)\Bigr) \, dy \\ 
& =\delta_x \phi_a^{\,\varepsilon}(x) - \frac{1}{2} h_\omega(ax) 
- \varepsilon \int_{0}^{a} h_\omega(xy) \delta_y  \phi_a^{\,\varepsilon}(y) \, dy.
\endaligned
\end{equation}
Substituting the right-hand side of \eqref{408} for the left-hand side of \eqref{407} 
and rearranging, we obtain
\begin{equation} \label{409}
\delta_x \phi_a^{\,\varepsilon}(x) - \varepsilon \int_{0}^{a} h_\omega(xy) \delta_y  \phi_a^{\,\varepsilon}(y) \, dy
= \left( \frac{1}{2} - \varepsilon a \phi_a^{\,\varepsilon}(a) \right) h_\omega(ax) +  a \frac{\partial}{\partial a} h_\omega(ax).
\end{equation}
Subtracting \eqref{409} with choice $-\varepsilon$ from \eqref{406} with $\varepsilon$, we obtain
\begin{equation} \label{410}
\aligned
\left\{ a \frac{\partial}{\partial a} \phi_a^{\,\varepsilon}(x)  - \delta_x \phi_a^{-\varepsilon}(x) \right\} 
+ \varepsilon \int_{0}^{a} & h_\omega(xy)  \left\{ a \frac{\partial}{\partial a} 
\phi_a^{\,\varepsilon}(y) - \delta_y  \phi_a^{-\varepsilon}(y) \right\} \, dy \\
& \qquad \qquad  = - \left( \frac{1}{2} + \varepsilon \mu(a) \right) h_\omega(ax), 
\endaligned
\end{equation}
where  
\begin{equation} \label{411}
\mu(a) = a \phi_a^+(a) +  a \phi_a^{-}(a).
\end{equation}
By \eqref{404}, Lemma \ref{lem_ad1} and \ref{lem_ad2}, the function $\mu(a)$ is continuous on $(1,\infty)$, 
which satisfies $\lim_{a\to 1^+}\mu(a)=0$, 
and is continuously differentiable on $(1,\infty) \setminus \{\sqrt{n}\,|\, n \in N\}$. 

Equality \eqref{410} shows that 
$a \frac{\partial}{\partial a} \phi_a^{\varepsilon}(x)  - \delta_x \phi_a^{-\varepsilon}(x)$ 
is a continuous solution of \eqref{405}. 
Hence, by comparing \eqref{405} with \eqref{410}, we obtain 
\begin{equation} \label{412}
a \frac{\partial}{\partial a} \phi_a^{\,\varepsilon}(x) - \delta_x \phi_a^{-\varepsilon}(x) 
 = - \left( \frac{1}{2} + \varepsilon \mu(a)  \right) \phi_a^{\,\varepsilon}(x) 
\quad (\varepsilon \in \{\pm 1\})
\end{equation}
by the uniqueness of solutions (Lemmas \ref{lem_405} and \ref{lem_406}). 
We use \eqref{412} in the  form
\begin{equation} \label{413}
\left( a \frac{\partial}{\partial a} +  \frac{1}{2} + \varepsilon\mu(a) \right) \phi_a^{\,\varepsilon}(x)
 = \delta_x \phi_a^{-\varepsilon}(x) \quad (\varepsilon \in \{\pm 1\}). 
\end{equation}
\smallskip

Now we introduce two special functions  
\begin{equation} \label{413_2}
\aligned
\tilde{A}_a(z) 
& := \frac{a^{iz}}{2} + \frac{\sqrt{a}}{2} \int_{a}^{\infty} \phi_a^{+}(x) \, x^{\frac{1}{2}+iz} \, \frac{dx}{x}
 = \frac{a^{iz}}{2} + \frac{\sqrt{a}}{2}\, {\mathsf F}_{1/2}(1-{\mathsf P}_{a})\phi_a^+(z), \\
-i\tilde{B}_a(z) 
&:= \frac{a^{iz}}{2} - \frac{\sqrt{a}}{2} \int_{a}^{\infty} \phi_a^{-}(x) \, x^{\frac{1}{2}+iz} \, \frac{dx}{x}
  = \frac{a^{iz}}{2} - \frac{\sqrt{a}}{2}\, {\mathsf F}_{1/2}(1-{\mathsf P}_{a})\phi_a^-(z)  
\endaligned
\end{equation}
for $\Im(z) \gg 0$ and $a>1$. These functions are defined as analytic functions for large $\Im(z)>0$ by integrals, since 
$\phi_a^{\pm}$ are continuous and have at most polynomial growth at $+\infty$ by \eqref{405} and the rough estimate
\begin{equation} \label{est_h}
\aligned
h_\omega(x) 
& = \frac{1}{2\pi}\int_{-U+ic}^{U+ic} \Theta_{\omega}(z) \, x^{-\frac{1}{2}-iz} \, dz +O(x^{c-\frac{1}{2}}U^{1-\omega}) 
\quad (c>1/2+\omega) \\
& = O(x^{c-\frac{1}{2}}U) + O(x^{c-\frac{1}{2}}U^{1-\omega}) = O(x^{\omega+\epsilon}).
\endaligned
\end{equation}
As shown below, $\tilde{A}_a $ and $\tilde{B}_a$ are analytically continuable to meromorphic functions in $\C$. 
We put it off a little and derive a differential system satisfied by $\tilde{A}_a $ and $\tilde{B}_a$.
Using \eqref{413}, we have 
\begin{equation*}
\aligned
\,&\left( a \frac{\partial}{\partial a}  + \mu(a) \right)\tilde{A}_a (z) 
 =\left( iz  + \mu(a) \right) \frac{a^{iz}}{2} 
+ \left( a \frac{\partial}{\partial a}  + \mu(a) \right)\frac{\sqrt{a}}{2} \int_{a}^{\infty} \phi_a^{+}(x) \, x^{\frac{1}{2}+iz} \, \frac{dx}{x} \\
& \quad  =\left( iz  + \mu(a) \right) \frac{a^{iz}}{2} - \frac{\sqrt{a}}{2} \phi_a^{+}(a) \, a^{\frac{1}{2}+iz}
 + \frac{\sqrt{a}}{2} \int_{a}^{\infty} \left( a \frac{\partial}{\partial a} + \frac{1}{2} + \mu(a) \right) \phi_a^{+}(x) \, x^{\frac{1}{2}+iz} \, \frac{dx}{x} \\
& \quad =\left( iz  + \mu(a) \right) \frac{a^{iz}}{2} 
- \frac{\sqrt{a}}{2} \phi_a^{+}(a) \, a^{\frac{1}{2}+iz} +   \frac{\sqrt{a}}{2} \int_{a}^{\infty} 
\left(x \frac{\partial}{\partial x} + \frac{1}{2}\right)\phi_a^{-}(x) \, x^{\frac{1}{2}+iz} \, \frac{dx}{x} \\
& \quad = iz   \frac{a^{iz}}{2} - iz \frac{\sqrt{a}}{2} \int_{a}^{\infty} \phi_a^{-}(x) \, x^{\frac{1}{2}+iz} \, \frac{dx}{x} 
 = z \, \tilde{B}_a (z)
\endaligned
\end{equation*}
for large $\Im(z)>0$, and then it holds for all $z \in \C$ by meromorphic continuation (below). 
We obtain a similar formula for $( a \frac{\partial}{\partial a}  + \mu(a))\tilde{B}_a $. 
As a result, we obtain the first order differential system 
\begin{equation} \label{414}
- \begin{bmatrix}a \frac{\partial}{\partial a}+  \mu(a) & 0 \\  0 &a \frac{\partial}{\partial a} -\mu(a) \end{bmatrix}
\begin{bmatrix} \tilde{A}_a (z) \\ \tilde{B}_a (z) \end{bmatrix}
= z \begin{bmatrix} 0 & -1 \\  1 & 0 \end{bmatrix}\begin{bmatrix} \tilde{A}_a(z) \\ \tilde{B}_a (z) \end{bmatrix} \quad (a>1).
\end{equation}
We extend the system to $a>0$ by taking the convention that 
\begin{equation} \label{415_1}
\mu(a) =0
\end{equation}
and 
\begin{equation} \label{415}
\aligned
\tilde{A}_a (z) & 
= \frac{a^{iz}}{2} + \frac{\sqrt{a}}{2}\int_{1/a}^{\infty} h_\omega(ax) \, x^{\frac{1}{2}+iz} \, \frac{dx}{x}
= \frac{1}{2}\Bigl( a^{iz} + \Theta_\omega(z) a^{-iz} \Bigr), \\
-i \tilde{B}_a (z) &= \frac{a^{iz}}{2} - \frac{\sqrt{a}}{2}\int_{1/a}^{\infty} h_\omega(ax) \, x^{\frac{1}{2}+iz} \, \frac{dx}{x}
= \frac{1}{2}\Bigl( a^{iz} - \Theta_\omega(z) a^{-iz} \Bigr) 
\endaligned
\end{equation}
for $0<a \leq 1$. Actually the convention \eqref{415_1} and \eqref{415} for $0<a\leq 1$ is compatible with Lemma \ref{lem_402} 
and the convention mentioned in the end of Section 4.1. 

For $a>0$, we define 
\begin{equation}\label{416}
\aligned
A_a(z) & = m(a) \, \xi\left(\frac{1}{2}+\omega-iz\right)\tilde{A}_a (z), \\ 
B_a(z) &= \frac{1}{m(a)}\,\xi\left(\frac{1}{2}+\omega-iz\right) \tilde{B}_a (z) 
\endaligned
\end{equation}
with  
\begin{equation} \label{417}
m(a) = \exp\left( \int_{1}^{a} \mu(b) \frac{db}{b}\right) \quad \left(\mu(a) = a \frac{d}{d a}\log m(a) \right)
\end{equation}
under  \eqref{415_1} and \eqref{415}. Note that $m(a)$ is real-valued by its definition. 
Then we can verify that system \eqref{414} implies that $(A_a,B_a)$ satisfies the canonical system
\begin{equation} \label{418}
-a \frac{\partial}{\partial a}\begin{bmatrix} X_a(z) \\ Y_a(z) \end{bmatrix}
= z \begin{bmatrix} 0 & -1 \\  1 & 0 \end{bmatrix}\begin{bmatrix} m(a)^{-2} & 0 \\ 0 & m(a)^{2} \end{bmatrix}
\begin{bmatrix} X_a(z) \\ Y_a(z) \end{bmatrix} \quad (0<a<\infty)
\end{equation}
by elementary ways. It is concluded that \eqref{418} is the canonical system of Theorem \ref{thm_4} 
if formula 
\begin{equation*}
m(a)=\frac{\det(1+{\mathsf H}_{\omega,a})}{\det(1-{\mathsf H}_{\omega,a})}
\end{equation*} 
is proved for $a>1$, since $\frac{\det(1+{\mathsf H}_{\omega,a})}{\det(1-{\mathsf H}_{\omega,a})}=1$ for $0<a\leq 1$ by Lemma \ref{lem_402}. 
This will follow from showing
\begin{equation} \label{419}
\aligned
\phi_a^{+}(a) & = \frac{d}{da}\log \det(1+{\mathsf H}_{\omega,a}), \\
\phi_a^{-}(a) & = - \frac{d}{da}\log \det(1-{\mathsf H}_{\omega,a})
\endaligned
\end{equation} 
by definition \eqref{411} and \eqref{417}. 
This is a well-known formula 
for an integral operator defined on a finite interval with a continuous kernel. 
In fact, it is proved by a way similar to the proof of Theorem 12 of Chapter 24 in \cite{MR1892228}. 
(This also holds for $0< a < 1$, since $\phi_a^\pm(a)=h_\omega(a^2)=0$ by the convention in the end of Section 4.1  
and $\log\det(1\pm{\mathsf H}_{\omega,a})=0$ by Lemma \ref{lem_402}.)
\medskip

For every fixed $0<a\leq 1$, $A_a$ and $B_a$ are real entire functions satisfying $A_a(-z)=A_a(z)$ and $B_a(-z)=-B_a(z)$, respectively, 
by \eqref{415}, \eqref{416} and functional equations $\xi(s)=\xi(1-s)$ and $\xi(s)=\overline{\xi(\bar{s})}$. 
Successively, we prove that $A_a $ and $B_a$ have these properties for $a>1$. 

%
%
\subsection{Meromorphic continuation and functional equations}
%
%
Under assumptions and notations of Section 4.3, we define 
\begin{equation} \label{420}
\aligned
\tilde{E}_a (z) &:= \tilde{A}_a (z) - i \tilde{B}_a (z) \\
& = a^{iz} + \frac{\sqrt{a}}{2} \int_{a}^{\infty} (\phi_a^{+}(x) - \phi_a^{-}(x) )\, x^{\frac{1}{2}+iz} \, \frac{dx}{x} 
 = a^{iz} + \frac{\sqrt{a}}{2} \, {\mathsf F}_{1/2}(1-{\mathsf P}_{a})(\phi_{a}^{+}-\phi_{a}^{-})(z) \\
\tilde{E}_a^{\,\ast}(z) & :=\tilde{A}_a (z) + i \tilde{B}_a (z) \\
& = \frac{\sqrt{a}}{2} \int_{a}^{\infty} (\phi_a^{+}(x) + \phi_a^{-}(x) )\, x^{\frac{1}{2}+iz} \, \frac{dx}{x} 
 =\frac{\sqrt{a}}{2}\, {\mathsf F}_{1/2}(1-{\mathsf P}_{a})(\phi_{a}^{+}+\phi_{a}^{-})(z)
\endaligned
\end{equation}
for $\Im(z) \gg 0$ and $a>1$. We deal with $A_a $, $B_a $ via $\tilde{E}_a (z)$ and $\tilde{E}_a^{\,\ast}(z)$. 

\begin{lemma} \label{lem_407}
Let $\omega>1$ and $a>1$. Define 
\begin{equation}\label{421}
\Psi_a(z)=\int_{a}^{\infty} (\phi_a^{+}(x) - \phi_a^{-}(x) )\, x^{\frac{1}{2}+iz} \, \frac{dx}{x}. 
\end{equation}
Then integral of \eqref{421} converges absolutely for $\Im(z)>0$ 
and converges in the $L^2$-sense on $\Im(z)=0$. 
Moreover $\Psi_a(z)$ is extended to a meromorphic function in $\C$ which is analytic in $\C^+$.   
\end{lemma}
\begin{proof}
By \eqref{405}, we have
\begin{equation} \label{422}
\phi_a^{+} - \phi_a^{-} = -{\mathsf H}_\omega {\mathsf P}_a( \phi_a^{+} + \phi_a^{-}), 
\end{equation}
where ${\mathsf P}_a( \phi_a^{+} + \phi_a^{-})$ has compact support in $[1/a,a]$. 
Therefore $\phi_a^{+} - \phi_a^{-}$ belongs to $L^2((\alpha,\infty),dx)$ for every $\alpha>0$. 
Hence integral \eqref{421} converges absolutely for $\Im(z)>0$ and defines a function of $H^2$ (\!\!\cite[Chap. II, \S10]{MR0005923}). 
Using \eqref{422}, we have
\begin{equation*}
\aligned
\Psi_a(z) 
&= {\mathsf F}_{1/2}(1-{\mathsf P}_a)(\phi_a^{+} - \phi_a^{-})(z) \\
&= - {\mathsf F}_{1/2}(1-{\mathsf P}_a){\mathsf H}_\omega {\mathsf P}_a( \phi_a^{+} + \phi_a^{-})(z) \\
&= - {\mathsf F}_{1/2}{\mathsf H}_\omega {\mathsf P}_a( \phi_a^{+} + \phi_a^{-})(z) 
+ {\mathsf F}_{1/2}{\mathsf P}_a{\mathsf H}_\omega {\mathsf P}_a( \phi_a^{+} + \phi_a^{-})(z)
\endaligned
\end{equation*}
for $\Im(z) \gg 0$. Here ${\mathsf P}_a( \phi_a^{+} + \phi_a^{-})$ and ${\mathsf P}_a{\mathsf H}_\omega {\mathsf P}_a( \phi_a^{+} + \phi_a^{-})$ 
have compact support in $(0,\infty)$. Therefore, we obtain
\begin{equation} \label{422_1}
\Psi_a(z) = - \Theta_\omega(z) {\mathsf F}_{1/2}{\mathsf P}_a( \phi_a^{+} + \phi_a^{-})(-z) 
+ {\mathsf F}_{1/2}{\mathsf P}_a{\mathsf H}_\omega {\mathsf P}_a( \phi_a^{+} + \phi_a^{-})(z), 
\end{equation}
where ${\mathsf F}_{1/2}{\mathsf P}_a( \phi_a^{+} + \phi_a^{-})(-z)$ 
and ${\mathsf F}_{1/2}{\mathsf P}_a{\mathsf H}_\omega {\mathsf P}_a( \phi_a^{+} + \phi_a^{-})(z)$ 
are entire functions. Hence $\Psi_a(z)$ is extended to a meromorphic function on $\C$, 
and is analytic in $\C^+$ by \eqref{422_1}, since $\Theta_\omega(z)$ is a meromorphic inner function in $\C^+$.  
\end{proof}

\begin{lemma} \label{lem_408} Let $\omega>1$ and $a>1$. 
Functions $\tilde{E}_a$ and $\tilde{E}_a^{\,\ast}$ of \eqref{420} are analytically continuable to meromorphic functions in $\C$ satisfying  
$\tilde{E}_a^{\,\ast}(z) = \Theta_\omega(z)\tilde{E}_a (-z)$, and they are analytic in $\C^+$. 
Moreover, both $\xi(\frac{1}{2}+\omega-iz)\tilde{E}_a (z)$ and $\xi(\frac{1}{2}+\omega-iz)\tilde{E}_a^{\,\ast}(z)$ are entire functions. 
\end{lemma}
\begin{proof}
We have 
\begin{equation} \label{424}
 {\mathsf H}_\omega {\mathsf P}_a( \phi_a^{+} - \phi_a^{-})(x) = 2h_\omega(ax) -\phi_a^{+}(x) - \phi_a^{-}(x)
\end{equation}
by \eqref{405}. Using \eqref{105}, \eqref{422}, and \eqref{422_1}, we have
\begin{equation*}
\aligned
\,&\Theta_{\omega}(z) \tilde{E}_a (-z) 
 = \Theta_{\omega}(z)  a^{-iz} + \Theta_{\omega}(z)\frac{\sqrt{a}}{2}\Psi_a(-z) \\
& \overset{\eqref{422_1}}{=} \Theta_{\omega}(z)  a^{-iz} 
- \Theta_{\omega}(z)\Theta_\omega(-z) \frac{\sqrt{a}}{2}{\mathsf F}_{1/2}{\mathsf P}_a( \phi_a^{+} + \phi_a^{-})(z) 
+ \frac{\sqrt{a}}{2}\Theta_{\omega}(z){\mathsf F}_{1/2}{\mathsf P}_a{\mathsf H}_\omega {\mathsf P}_a( \phi_a^{+} + \phi_a^{-})(-z) \\
& \overset{\eqref{105}}{=} \Theta_{\omega}(z)  a^{-iz} 
-  \frac{\sqrt{a}}{2}{\mathsf F}_{1/2}{\mathsf P}_a( \phi_a^{+} + \phi_a^{-})(z) 
+ \Theta_{\omega}(z)\frac{\sqrt{a}}{2}{\mathsf F}_{1/2}{\mathsf P}_a{\mathsf H}_\omega {\mathsf P}_a( \phi_a^{+} + \phi_a^{-})(-z) \\
& \overset{\eqref{422}}{=} \Theta_{\omega}(z)  a^{-iz} 
-  \frac{\sqrt{a}}{2}{\mathsf F}_{1/2}{\mathsf P}_a( \phi_a^{+} + \phi_a^{-})(z) 
- \Theta_{\omega}(z)\frac{\sqrt{a}}{2}{\mathsf F}_{1/2}{\mathsf P}_a(\phi_a^{+} - \phi_a^{-})(-z) 
\endaligned
\end{equation*}
for $z \in \C$, since 
${\mathsf P}_a(\phi_a^{+} \pm \phi_a^{-})$, and 
${\mathsf P}_a{\mathsf H}_\omega {\mathsf P}_a( \phi_a^{+} + \phi_a^{-})$ have compact support. 
Further, by Proposition \ref{prop_1}, Lemma \ref{lem_401}, and \eqref{424}, we have 
\begin{equation*}
\aligned
\Theta_{\omega}(z) \tilde{E}_a (-z) 
& = \Theta_{\omega}(z)  a^{-iz} 
-  \frac{\sqrt{a}}{2}{\mathsf F}_{1/2}{\mathsf P}_a( \phi_a^{+} + \phi_a^{-})(z) 
- \Theta_{\omega}(z)\frac{\sqrt{a}}{2}{\mathsf F}_{1/2}{\mathsf P}_a(\phi_a^{+} - \phi_a^{-})(-z) \\
& \overset{\eqref{401}}{=} \Theta_{\omega}(z)  a^{-iz} 
-  \frac{\sqrt{a}}{2}{\mathsf F}_{1/2}{\mathsf P}_a( \phi_a^{+} + \phi_a^{-})(z) 
- \frac{\sqrt{a}}{2}{\mathsf F}_{1/2}{\mathsf H}_\omega{\mathsf P}_a(\phi_a^{+} - \phi_a^{-})(z)\\
& \overset{\eqref{424}}{=} \Theta_{\omega}(z)  a^{-iz} 
-  \frac{\sqrt{a}}{2}{\mathsf F}_{1/2}{\mathsf P}_a( \phi_a^{+} + \phi_a^{-})(z) \\
& \qquad - \sqrt{a}\,{\mathsf F}_{1/2}(h_\omega(ax))(z)
+ \frac{\sqrt{a}}{2}{\mathsf F}_{1/2}(\phi_a^{+} + \phi_a^{-})(z) \\
& \overset{\eqref{206}}{=} \Theta_{\omega}(z)  a^{-iz} 
-  \frac{\sqrt{a}}{2}{\mathsf F}_{1/2}{\mathsf P}_a( \phi_a^{+} + \phi_a^{-})(z) - \Theta_{\omega}(z)  a^{-iz} 
+ \frac{\sqrt{a}}{2}{\mathsf F}_{1/2}(\phi_a^{+} + \phi_a^{-})(z) \\
& = \frac{\sqrt{a}}{2}{\mathsf F}_{1/2}(1-{\mathsf P}_a)(\phi_a^{+} + \phi_a^{-})(z) = \tilde{E}_a^{\,\ast}(z) 
\endaligned
\end{equation*}
for $\Im(z) \gg 0$, since $\phi_a^{+} + \phi_a^{-}$ is identically zero on $(0,1/a)$ and has polynomial growth at $x=+\infty$. 
Hence $\tilde{E}_a^{\,\ast}(z) = \Theta_\omega(z)\tilde{E}_a (-z)$ for $\Im(z) \gg 0$.  
By Lemma \ref{lem_407}, $\tilde{E}_a (z)$ is meromorphic in $\C$, therefore, 
$\tilde{E}_a^{\,\ast}(z)$ is also analytically continuable to a meromorphic function in $\C$. 
Moreover, $\tilde{E}_a (z) =  \Theta_\omega(z)\text{(entire)}+\text{(entire)}$ from the proof of Lemma \ref{lem_407}. 
Thus 
\begin{equation*}
\tilde{E}_a^{\,\ast}(z)=\Theta(z)\tilde{E}_a (-z) =  \text{(entire)}+\Theta_\omega(z)\text{(entire)} 
\end{equation*}
by \eqref{105}, and hence $\tilde{E}_a^{\,\ast}(z)$ is analytic in $\C^+$. 
Simultaneously, these equalities show that 
$\xi(\frac{1}{2}+\omega-iz)\tilde{E}_a (z)$ and $\xi(\frac{1}{2}+\omega-iz)\tilde{E}_a^{\,\ast}(z)$ are entire 
by definition of $\Theta_\omega(z)$.  
\end{proof}

Lemma \ref{lem_408} implies the following immediately. 

\begin{lemma} \label{lem_409} 
Let $\omega>1$ and $a>1$. 
Then $\tilde{A}_a (z)$ and $\tilde{B}_a (z)$ are analytically continuable to meromorphic functions in $\C$, and they are analytic in $\C^+$. 
Also, $A_a(z)$ and $B_a(z)$ are both entire functions. 
In addition, we have functional equations 
\begin{equation*}
\aligned
\Theta_{\omega}(z)\tilde{A}_a (-z) &=\tilde{A}_a (z), \qquad \Theta_{\omega}(z) \tilde{B}_a (-z)  = - \tilde{B}_a (z), \\
A_a (-z) & = A_a (z), \qquad B_a (-z) = - B_a (z).
\endaligned
\end{equation*}
\end{lemma}
\begin{proof} 
We have 
$2\tilde{A}_a =\tilde{E}_a +\tilde{E}_a^{\,\ast}$ and $-2i \tilde{B}_a =\tilde{E}_a - \tilde{E}_a^{\,\ast}$ by definition \eqref{420}.  
Therefore, they are analytically continuable to meromorphic function in $\C$ and satisfy   
$2\tilde{A}_a (z) = \tilde{E}_a (z) + \Theta_{\omega}(z)\tilde{E}_a (-z)$ and 
$-2i \tilde{B}_a (z) = \tilde{E}_a (z) - \Theta_{\omega}(z) \tilde{E}_a (-z)$ by Lemma \ref{lem_408}. 
That imply the functional equations stated in the lemma. 
Other things are consequences of Lemma \ref{lem_408}. 
\end{proof}

\begin{lemma} \label{lem_411}
Let $\omega>1$ and $a>1$. 
Then $A_a(z)$ and $B_a(z)$ are real entire functions.
\end{lemma}

\begin{proof} 
At first, we note that if $F(z)={\mathsf F}_{1/2}(f(x))(z)$ for $\Im(z) \gg 0$, then 
$F^\sharp(z)={\mathsf F}_{1/2}(x^{-1}\overline{f(x^{-1}}))(z)$ and $F(-z)={\mathsf F}_{1/2}(x^{-1}f(x^{-1}))(z)$ for $\Im(z) \ll 0$. 
Therefore, if $f(x)$ (resp. $if(x)$) is real-valued, 
$F(z)$ is analytically continued to a meromorphic function in $\C$, 
and $F(-z)=F(z)$ (resp. $F(-z)=-F(z)$), then $F^\sharp(z)=F(z)$ holds for $z \in \C$. 
Let 
\[
\phi(x) := \frac{1}{2}\frac{d}{dx}\Bigl( x^{2}\frac{d}{dx}\theta(x^2) \Bigr) =
2 \sum_{n=1}^{\infty} \bigl( 2\pi^2 n^4 x^4 - 3 \pi n^2 x^2 \bigr)
\exp(-\pi n^2 x^2).
\]
Then $\phi(1/x)=x\phi(x)$ and $\xi(s)=\int_{0}^{\infty} \phi(x) x^{s} \frac{dx}{x}$ for every $s \in \C$ (\cite[\S10.1]{MR882550}). 
Hence $\xi(\frac{1}{2}+\omega-iz) = {\mathsf F}_{1/2}(x^{-\omega}\phi(x))(z)$ for $z \in \C$. 
On the other hand, by \eqref{413_2}, 
$\tilde{A}_a(z)={\mathsf F}_{1/2}(\frac{\sqrt{a}}{2}(\delta_a+(1-{\mathsf P}_a)\phi_a^+))(z)$ 
and 
$\tilde{B}_a(z)={\mathsf F}_{1/2}(\frac{i\sqrt{a}}{2}(\delta_a-(1-{\mathsf P}_a)\phi_a^-))(z)$
for $\Im(z) \gg 0$ , 
where $\delta_a(x)$ is the Dirac delta-function at $x=a$. 
Therefore $A_a(z)={\mathsf F}_{1/2}(f^+(x))(z)$ and $B_a(z)={\mathsf F}_{1/2}(f^-(x))(z)$ for  
\[
\aligned
f^+(x):&=\frac{m(a)\sqrt{a}}{2}\int_{0}^{\infty} (x/y)^{-\omega}\phi(x/y)(\delta_a(y)+(1-{\mathsf P}_a)\phi_a^+(y)) \frac{dy}{y} \\
&=\frac{m(a)x^{-\omega}}{2}\left(a^{\omega-\frac{1}{2}}\phi(x/a) + \sqrt{a}\int_{a}^{\infty} \phi(x/y)\phi_a^+(y) \, y^{\omega-1} \, dy \right),
\endaligned
\]
\[
\aligned
f^-(x):&=\frac{i\sqrt{a}}{2m(a)}\int_{0}^{\infty} (x/y)^{-\omega}\phi(x/y)(\delta_a(y)-(1-{\mathsf P}_a)\phi_a^-(y)) \frac{dy}{y} \\
&=\frac{ix^{-\omega}}{2m(a)}\left(a^{\omega-\frac{1}{2}}\phi(x/a) - \sqrt{a}\int_{a}^{\infty} \phi(x/y)\phi_a^-(y) \, y^{\omega-1}\,dy \right) 
\endaligned
\]
if $\Im(z) \gg 0$. 
Here $f^+(x)$ and $if^-(x)$ are both real-valued, since $m(a)$ is real, and $\phi(x)$, $\phi_a^\pm(x)$ are real-valued. 
In addition, $A_a(-z)=A_a(z)$ and $B_a(-z)=-B_a(z)$ for $z \in \C$ by Lemma \ref{lem_409}. 
Hence $A_a^\sharp=A_a$ and $B_a^\sharp=B_a$.
\end{proof}

Now we complete the proof of Theorem \ref{thm_4} (1), (2). 
The remaining assertion is (3). 
In order to prove it, we show the following lemma.  

\begin{lemma} \label{lem_410} 
Let $\omega>1$. Then   
\begin{equation}\label{425}
\lim_{a \to 1^+} A_a(z) = A^\omega(z), \qquad \lim_{a \to 1^+} B_a(z) = B^\omega(z)
\end{equation}
hold uniformly on every compact subset in $\C$. 
\end{lemma}
\begin{proof}
By \eqref{402} and \eqref{404_1}, $\phi_a^\pm(x) \to h_\omega(x)$ uniformly on $[1/2,3/2]$ as $a \to 1^+$. 
Therefore, by \eqref{420} and \eqref{422_1}, $\tilde{E}_a(z)$ converges to a meromorphic function in $\C$ uniformly on every compact subset in $\C$ as $a \to 1^+$, 
since  ${\mathsf P}_a( \phi_a^{+} + \phi_a^{-})(-z)$ 
and ${\mathsf P}_a{\mathsf H}_\omega {\mathsf P}_a( \phi_a^{+} + \phi_a^{-})(z)$ 
both have support in $[1/a,a]$. 
On the other hand, we have 
\begin{equation} \label{426}
\phi_a^{+}(x) - \phi_a^{-} (x)
= - \int_{1/a}^{a} h_\omega(xz) (\phi_a^{+}(z) + \phi_a^{-} (z)) \, dz
\end{equation}
by \eqref{405}, since $\phi_a^{\,\pm}(x)=0$ for $0<x<1/a$.   
Multiplying by $x^{-v}$ on both sides of \eqref{426}, 
and then tending $a \to 1^+$, we have 
\begin{equation*}
\lim_{a \to 1^+} (\phi_a^{+}(x) - \phi_a^{-} (x))x^{-v} = 0 
\end{equation*}
uniformly on $(1,\infty)$ if $v>0$ is large, since $h_\omega$ is of polynomial growth at $+\infty$. 
Hence $\lim_{a \to 1^+} \tilde{E}_a (z)=1$ uniformly on every compact subset in $\Im(z)>v$. 
As a consequence $\lim_{a \to 1^+} \tilde{E}_a (z)=1$, and 
\begin{equation*}
\lim_{a \to 1^+}\tilde{A}_a (z) = \frac{1}{2}(1 + \Theta_\omega(z)), \quad 
\lim_{a \to 1^+}\tilde{B}_a (z) = \frac{i}{2}(1 - \Theta_\omega(z))
\end{equation*}      
uniformly on every compact subset in $\C$. 
Multiplying by $\xi(\frac{1}{2}+\omega-iz)$ on both sides of these equalities, 
we obtain  \eqref{425} by \eqref{404} and \eqref{416}.  
\end{proof}

\noindent
By definition, we have $m(a)=1$ and 
\[
\aligned
A_a(z) &= \xi\left(\frac{1}{2}+\omega-iz\right)a^{iz} + \xi\left(\frac{1}{2}-\omega-iz\right)a^{-iz}, \\ 
B_a(z) &= \xi\left(\frac{1}{2}+\omega-iz\right)a^{iz} - \xi\left(\frac{1}{2}-\omega-iz\right)a^{-iz}
\endaligned
\]
for $0<a\leq 1$. 
This shows $(A_1,B_1)=(A^\omega,B^\omega)$ and 
$
\displaystyle{\lim_{a \to 1^-} (A_a(z),B_a(z)) = (A^\omega(z),B^\omega(z))}
$ 
uniformly on every compact subset in $\C$. 
Together with Lemma \ref{lem_410},  we obtain Theorem \ref{thm_4} (3), and hence we complete the proof of Theorem \ref{thm_4}.
%
%
\section{Comments on the validity of Theorem \ref{thm_4}}
%
%
In this section, we comment on a range of $\omega>0$ where Theorem \ref{thm_4} is expected to be extended. 
There might be three levels of difficulties at least: (i) $\omega>1/2$, (ii) $\omega=1/2$, (iii) $0<\omega<1/2$.
\medskip

It is natural to expect that Theorem \ref{thm_4} is proved unconditionally for the range (i) as mentioned after Theorem \ref{thm_4}. 
In fact, all lemmas in Section 4.1 are already proved for $\omega>1/2$. 
Therefore, the remaining problems are a proof of the differentiability of $\phi_a^\varepsilon(x)$ with respect to $x$ and $a$, 
and formula of $m(a)$ by determinants.  
However, if we understand partial derivatives 
$\frac{\partial}{\partial x}\phi_a^\varepsilon(x)$ and $\frac{\partial}{\partial a}\phi_a^\varepsilon(x)$ 
in the sense of distributions as in Burnol~\cite{B1}, 
and if we use the theory of Fredholm determinants for $L^2$-kernels (\cite[Chap. VI]{MR0104991}), 
then most of Section 4.3 and 4.4 have reasonable meaning, 
and we may obtain Theorem \ref{thm_4} for $\omega>1/2$. 
This way is plausible, and must be carried out after a suitable preparation for the theory of distributions. 

The case (ii) have more difficulties, because the kernel of ${\mathsf H}_{\omega,a}$ is no longer Hilbert-Schmidt type. 
However, $\Theta_\omega(z)$ is still inner function in $\C^+$ unconditionally. 
Therefore, problems may be restricted to the theory of integral operators, its determinants, and the theory of integral equations only as well as the case (i). 
See the later half of comments on (iii) below.  
\medskip

It is easily predicted that it is very hard to generalize Theorem \ref{thm_4} to the range (iii) unconditionally. 
A reason of difficulties is that problems of arithmetic and analysis are mixed in this range. 
However, if we assume RH, the function $\Theta_\omega$ is inner in $\C^+$ for every $\omega>0$, 
and hence remaining problems may be restricted to the theory of integral operators and the theory of integral equations only. 
Such analytic problems may be solved without essential difficulties. 

In fact, if $\Theta_\omega(z)$ is an inner function in $\C^+$, 
${\mathsf H}_{\omega}$ is extended to an isometry on $L^2((0,\infty),dx)$ by Lemma \ref{lem_401}. 
On the other hand, ${\mathsf H}_{\omega,a}$ is a compact operator on $L^2((0,a),dx)$ even for (iii) (and (ii)) 
because its kernel is a sum of finitely many weakly singular kernels. 
Therefore, in particular, the Fredholm alternative holds. 
Hence we may obtain reasonable generalization of results in Section 4.1 for $\omega>0$ under RH,  
and then throughout distribution theoretic dealing of Section 4.3 and 4.4, 
we may arrive at the generalization of Theorem \ref{thm_4} for the range (iii) (and (ii)) under RH. 
In this strategy, it is necessary to note that $\phi_a^\varepsilon(x)$ have some possible singularities, 
which affect definition \eqref{411} of $\mu(a)$ and definition \eqref{417} of $m(a)$, 
and that the definition of determinants $\det(1 \pm {\mathsf H}_{\omega,a})$ should be changed 
as in K\"onig~\cite{MR568991}. 
\medskip

We leave a justification of the above argument for a future study.  

%
\appendix
%
%
\section{}

Suppose that $\Theta_\omega(z)$ is an inner function in $\C^+$. 
(It holds unconditionally for $\omega \geq 1/2$, and also for $0<\omega<1/2$ under RH.)  
Then it defines the reproducing kernel Hilbert space $K(\Theta_\omega)$ 
which is isomorhic to the de Branges space $B(E^\omega)$ (see Section 1.3 and 1.4). 
According to the theory of de Branges~\cite{MR0229011}, 
the structure of $B(E^\omega)$ is determined by associated canonical system, 
which was described in terms of the shifted Fourier inversion $h_\omega(x)$ of $\Theta_\omega(z)$ under the restriction $\omega>1$.   
On the other hand, the structure of $B(E^\omega)$ is also determined by the reproducing kernel of $K(\Theta_\omega)$:
\begin{equation*}
K_\omega(z,w) = \frac{1}{2\pi i} \, \frac{1 - \overline{\Theta_\omega(z)}\Theta_\omega(w)}{\bar{z} - w} \quad (z,w \in \C^+)  
\end{equation*}
(see Section 1.3). We find that $K_\omega(0,\ast)$ belongs to $L^2(\R)$ by \eqref{106} and \eqref{107}, 
and thus its shifted Fourier inversion ${\mathsf F}_{1/2}^{-1}K_\omega(0,\ast)$ belongs to $L^2((0,\infty),dx)$. 

However, if we obtain ${\mathsf F}_{1/2}^{-1}K_\omega(0,\ast)$ explicitly enough, 
we may define ${\mathsf F}_{1/2}^{-1}K_\omega(0,\ast)$ regardless whether $\Theta_\omega(z)$ is an inner function in $\C^+$. 
In fact, it is carried out by using the weighted summatory function $h_\omega^{\langle 1 \rangle}(x)$ defined below.  
Then sufficient or equivalent conditions for $\Theta_\omega(z)$ to be an inner function in $\C^+$
are given in terms of $h_\omega^{\langle 1 \rangle}(x)$ 
as in Theorem \ref{thm_2}. 
This is the main result in the appendix. 

The function $h_\omega^{\langle 1 \rangle}(x)$ is not only directly related to RH via the innerness of $\Theta_\omega(z)$, 
but also directly related to the operator ${\mathsf H}_\omega$ (Theorem \ref{thm_5}). 
The above discussion clarifies the meaning of a part of functions studied in ~\cite{Su} 
(see the remark after Theorem \ref{thm_3}).   

\subsection{Notation and Results}

Let $B(z;p,q)$ be the incomplete beta function defined by 
\begin{equation*}
B(z;p,q) = \int_{0}^{z} x^{p-1}(1-x)^{q-1} \, dx 
\quad (0 \leq z \leq 1, \, \Re(p)>0, \, \Re(q)>0).
\end{equation*}
%
We use the notation 
\begin{equation*}
\beta(z;p,q) := B(p,q) - B(z;p,q) =  \int_{z}^{1} x^{p-1}(1-x)^{q-1} \, dx,    
\end{equation*}
and understand that $\beta(z;p,q)$ is defined by the integral on the right-hand side 
if $\Re(p) \leq 0$, $\Re(q)>0$, and $0<z<1$. For example, $g_\omega$ of \eqref{201} 
can be written as  
\begin{equation*}
\aligned
g_\omega(x) 
&= \frac{2\pi^\omega}{\Gamma(\omega)}\left[ x^{2-\omega} (1-x^2)^{\omega-1}
- \omega x^{\omega-1} \beta\left(x^2, \frac{3 - 2 \omega}{2},  \omega \right)
\right]. 
\endaligned
\end{equation*}
We define the real-valued function $g_\omega^{\langle 1 \rangle}$ on $(0,\infty)$ by 
\begin{equation}\label{a202}
g_\omega^{\langle 1 \rangle}(x) :=
\int_{x}^{1} \sqrt{\frac{y}{x}} \, g_\omega(y) \, \frac{dy}{y} 
\end{equation}
for $0<x<1$, and $g_\omega^{\langle 1 \rangle}(x)=0$ for $x>1$. 
Then we have 
\begin{equation*}
\aligned
~&g_\omega^{\langle 1 \rangle}(x) = \\
& 
\begin{cases}
~\displaystyle{\frac{4\omega}{2\omega-1} 
\frac{\pi^\omega}{\Gamma(\omega)}
\left\{ x^{\omega-1}
\, \beta\left(x^2, \frac{3-2\omega}{2},\omega \right) - \frac{2\omega+1}{4\omega} \, x^{-1/2} \, 
\beta\left(x^2, \frac{5-2\omega}{4},\omega \right)
\right\}}, & \omega\not=1/2, \\[8pt]
~\displaystyle{\frac{2}{\sqrt{x}} \left( 2 \sqrt{1-x^2} + \log x  - \log(1 + \sqrt{1 - x^2}) \right) }, 
& \omega=1/2
\end{cases}
\endaligned
\end{equation*}
for $0 < x < 1$ by elementary ways. 
Using $g_\omega^{\langle 1 \rangle}$ and  $c_\omega(n)$ of \eqref{203}, 
we define the real-valued function $h_\omega^{\langle 1 \rangle}$ on $(0,\infty)$ by 
\begin{equation} \label{a204}
h_\omega^{\langle 1 \rangle}(x) = 
\displaystyle{\frac{1}{x} \sum_{n=1}^{\lfloor x \rfloor} c_\omega(n) \, g_\omega^{\langle 1 \rangle}\left(\frac{n}{x}\right)}
\end{equation}
for $x>1$, and $h_\omega^{\langle 1 \rangle}(x)=0$ for $0<x<1$. 
Then $h_\omega^{\langle 1 \rangle}$ is well-defined on $(0,\infty)$ and has a support in $[1,\infty)$ as well as $h_\omega$. 
We also have 
\begin{equation} \label{a205}
h_\omega^{\langle 1 \rangle}(x) =  \int_{1}^{x} \sqrt{\frac{y}{x}} \, h_\omega(y) \frac{dy}{y}
\end{equation}
for $x>1$ by definition \eqref{a202}. 
The function $h_\omega^{\langle 1 \rangle}$ is related to $K_\omega(z,w)$ and $\Theta_\omega$ as follows. 
\begin{theorem} \label{thm_3}
Let $\omega>0$. 
\begin{enumerate}
\item[(1)] If $\Theta_\omega(z)$ is an inner function in $\C^+$, we have
\begin{equation*}
F_{1/2}^{-1}K_\omega(0,\ast)(z)
= \frac{1}{2\pi}\left(x^{-\frac{1}{2}}{\mathbf 1}_{(1,\infty)}(x) - h_{\omega}^{\langle 1 \rangle}(x) \right), 
\end{equation*}
where ${\mathbf 1}_{(1,\infty)}$ is the characteristic function of $(1,\infty)$. 
\item[(2)] $\Theta_{\omega}(z)$ is an inner function in $\C^+$ if and only if 
$(x^{-\frac{1}{2}}{\mathbf 1}_{(1,\infty)}(x)-h_{\omega}^{\langle 1 \rangle}(x))$ belongs to $L^2((1,\infty),dx)$. 
\item[(3)] Assume that there exists $x_\omega \geq 1$ such that 
$h_\omega^{\langle 1 \rangle}(x)$ has a single sign for every $x \geq x_\omega$. 
Then $\Theta_\omega(z)$ is an inner function in $\C^+$. 
\item[(4)] Assume that $\lim_{x \to \infty}\sqrt{x}\,h_{\omega}^{\langle 1 \rangle}(x)$ exists. 
Then $\Theta_{\omega}(z)$ is an inner function in $\C^+$. 
\item[(5)] Assume that $\Theta_\omega(z)$ is an inner function in $\C^+$ for all $\omega>0$. 
Then we have 
\begin{equation*}
\sqrt{x} h_{\omega}^{\langle 1 \rangle}(x) = 1 + o(1)
\end{equation*} 
as $x \to +\infty$ for all $\omega >0$. 
\end{enumerate}
\end{theorem} 
\begin{remark} The case $\lim_{x \to \infty}\sqrt{x}\,h_{\omega}^{\langle 1 \rangle}(x)=0$ 
is allowed in (4), though it does not hold by (5) if RH holds for $\zeta(s)$.   
\end{remark}
\begin{remark} 
Functions $h_{\omega}(x)$ of \eqref{204} and $h_{\omega}^{\langle 1 \rangle}(x)$ of \eqref{a204} 
were introduced and studied in \cite{Su} for more general $L$-functions, 
but notation is different a little. 
The function $h_{\omega}^{\langle 1 \rangle}(x)$ (resp. $h_{\omega}(x)$) 
was denoted by $x^{-\frac{1}{2}}h_{{\mathbf 1},\omega}^{\langle 1 \rangle}(x)$ 
(resp. $x^{-\frac{1}{2}}h_{{\mathbf 1},\omega}^{\langle 0 \rangle}(x)$) in \cite{Su}. 
\end{remark}

\begin{theorem} \label{thm_5} 
Suppose that $\Theta_\omega(z)$ is an inner function in $\C^+$. 
We define 
\[
(\tilde{\mathsf H}_\omega f)(x)=\sqrt{x} \frac{d}{dx}\sqrt{x} \int_{0}^{\infty}  h_\omega^{\langle 1 \rangle}(xy) \, f(y) \, dy
\]
for compactly supported smooth functions $f$. Then $\tilde{\mathsf H}_\omega f$ belongs to $L^2((0,\infty),dx)$, 
and $f \mapsto \tilde{\mathsf H}_\omega f$ is extended to the isometry on $L^2((0,\infty),dx)$
satisfying $\tilde{\mathsf H}_\omega f = {\mathsf H}_\omega f$. 
\end{theorem}
%
%
\subsection{Proof of Theorem \ref{thm_3}}
%
%
We prove each statement of Theorem \ref{thm_3} separately. At first, we note the following:  
\begin{proposition} \label{prop_a1} For $\omega>0$ and $\Im(z)>1/2+\omega$, we have 
\begin{equation} \label{a207}
\int_{0}^{\infty} h_{\omega}^{\langle 1 \rangle}(x) \, x^{\frac{1}{2}+iz} \, \frac{dx}{x} 
 = \frac{i}{z} \, \Theta_\omega(z),
\end{equation}
where the integral converges absolutely. 
\end{proposition}
\begin{proof}
This is proved by a way similar to the proof of Proposition \ref{prop_1} in Section 3.1 (see also Lemma 4.2 of \cite{Su}). 
\end{proof}

\noindent
{\bf (3):} By the assumption and a theorem of Landau (e.g. Widder \cite[Chap.II,\S5]{MR0005923}), 
the integral in \eqref{a207} converges for $\Im(z) > v_0$, 
where $i v_0$ is the first pure imaginary singularity of $\Theta_\omega(z)/z$. 
On the other hand, $\Theta_\omega(z)$ has no singularities on the imaginary axis, 
because it is known that $\xi(s)$ has no real zeros. 
Hence $\Theta_\omega(z)$ is regular in $\C^+$. 
It implies that $\Theta_\omega(z)$ is an inner function in $\C^+$ by a way similar to the proof of Lemma \ref{lem_301}. \hfill $\Box$
\medskip

\noindent
{\bf (1) and (2):} Suppose that $\Theta_\omega(z)$ is an inner function in $\C^+$. We have 
\begin{equation*}
h_\omega^{\langle 1 \rangle}(x) = \frac{1}{2\pi}\lim_{U\to\infty}\int_{-U+ic}^{U+ic} \frac{\Theta_{\omega}(z)}{-iz} \, x^{-\frac{1}{2}-iz} \, dz 
\quad (c>1/2+\omega)
\end{equation*}
for $x>1$ by the Mellin inversion formula (e.g. \cite[Theorem 28]{MR942661}),  
since the integral in \eqref{206} converges absolutely for $\Im(z)>1/2+\omega$ 
and $h_\omega^{\langle 1 \rangle}(x)$ is in $C^1(1,\infty)$. 
By the Stirling formula, we have $\Theta_\omega(u+iv) \ll_{\omega,v} u^{-\omega}$ for a fixed $v>1/2+\omega$. 
Therefore 
\begin{equation*}
h_\omega^{\langle 1 \rangle}(x) = \frac{1}{2\pi}\int_{-U+ic}^{U+ic} \frac{\Theta_{\omega}(z)}{-iz} \, x^{-\frac{1}{2}-iz} \, dz +O(x^{c-\frac{1}{2}}U^{-\omega}) 
\quad (c>1/2+\omega).
\end{equation*}
Using the well-known formula 
\begin{equation*}
x^{-\frac{1}{2}} = \frac{1}{2\pi}\int_{-U+ic}^{U+ic} \frac{1}{-iz} \, x^{-\frac{1}{2}-iz} \, dz +O(x^{c-\frac{1}{2}}(\log x)^{-1}U^{-1})
\end{equation*}
for $x>1$ and large $U>1$, we have 
\begin{equation*}
x^{-\frac{1}{2}} - h_\omega^{\langle 1 \rangle}(x) 
= \frac{1}{2\pi}\int_{-U+ic}^{U+ic} \frac{1-\Theta_{\omega}(z)}{-iz} \, x^{-\frac{1}{2}-iz} \, dz +O(x^{c-\frac{1}{2}}U^{-\omega}) + 
O(x^{c-\frac{1}{2}}(\log x)^{-1}U^{-1}).
\end{equation*}
Here the integrand $(1-\Theta_\omega(z))/z$ is bounded on $\C^+ \cup \R$. 
Thus the residue theorem gives
\begin{equation*}
x^{-\frac{1}{2}} - h_\omega^{\langle 1 \rangle}(x) 
= \frac{1}{2\pi}\int_{-U}^{U} \frac{1-\Theta_{\omega}(z)}{-iz} \, x^{-\frac{1}{2}-iz} \, dz +O(x^{c-\frac{1}{2}}U^{-\omega}) + 
O(x^{c-\frac{1}{2}}(\log x)^{-1}U^{-1}),
\end{equation*}
since integrals on $\int_{\pm U+i0}^{\pm U+ic}$ are bounded by $x^{c-\frac{1}{2}}(\log x)^{-1}U^{-1}$.
Tending $U$ to $+\infty$ for fixed $x>1$, we have 
\begin{equation} \label{303}
x^{-\frac{1}{2}} - h_\omega^{\langle 1 \rangle}(x) 
= \frac{1}{2\pi}\int_{-\infty}^{\infty} \frac{1-\Theta_{\omega}(u)}{-iu} \, x^{-\frac{1}{2}-iu} \, du 
\quad (x>1). 
\end{equation}
This implies that $x^{-\frac{1}{2}}{\mathbf 1}_{(1,\infty)}(x) - h_\omega^{\langle 1 \rangle}(x)$ 
belongs to $L^2((1,\infty),dx)$, since $(1-\Theta_{\omega}(u))/u$ belongs to $L^2(\R)$ by \eqref{106} and \eqref{107}.  
In addition, \eqref{303} implies (1). 

Suppose that $x^{-1/2}{\mathbf 1}_{(1,\infty)}-h_{\omega}^{\langle 1 \rangle}$ belongs to $L^2((1,\infty),dx)$. 
Then the integral
\begin{equation}\label{304}
\int_{0}^{\infty} \Bigl[x^{-\frac{1}{2}}{\mathbf 1}_{(1,\infty)}(x)-h_{\omega}^{\langle 1 \rangle}(x) \Bigr] \, x^{\frac{1}{2}+iz} \,  \frac{dx}{x}  
\end{equation}
converges on the real line in $L^2$-sense, and converges absolutely for $\Im(z)>0$ (\cite[Chap.II, \S10]{MR0005923}). 
Hence integral \eqref{304} defines an analytic function in $\C^+$. 
By Proposition \ref{prop_a1}, integral \eqref{304} is equal to $(1-\Theta_\omega(z))/(iz)$ for $\Im(z)>1/2+\omega$. 
Hence we find that $\Theta_\omega(z)$ is an analytic function in $\C^+$, 
and it implies that $\Theta_\omega(z)$ is an inner function in $\C^+$ as well as the proof of (3). 
\hfill $\Box$
\medskip

\noindent
{\bf (4):} By formula \eqref{a205}, the assumption implies that the integral of \eqref{206} converges at $z=0$ in the sense 
\begin{equation*}
\lim_{T \to \infty} \int_{1}^{T} h_\omega(x) \, x^{\frac{1}{2}+i0} \, \frac{dx}{x}. 
\end{equation*}
This implies that the integral of \eqref{206} converges for $\Im(z)>0$, 
and defines an analytic function in $\C^+$ (\!\!\cite[Chap.II, \S1]{MR0005923}). 
Hence $\Theta_\omega(z)$ is an inner function in $\C^+$ by \eqref{106}. 
\hfill $\Box$
\medskip

\noindent
{\bf (5):} If $\Theta_\omega(z)$ is an inner function in $\C^+$ for all $\omega>0$, ${\rm RH}(A^\omega)$ holds for all $\omega>0$. 
Hence RH holds by Proposition \ref{prop_basic}. 
Then we obtain (5) by a way similar to the proof of Theorem 2.3 (2-b) in \cite{Su}. 
\hfill $\Box$
%
%
\subsection{Proof of Theorem \ref{thm_5}}
%
%
Suppose that $\Theta_\omega(z)$ is an inner function in $\C^+$. 
Let $f$ be a compactly supported smooth function. 
Put $F={\mathsf F}_{1/2}f$ and define 
\[
g(x)=\frac{1}{2\pi} \int_{c-i\infty}^{c+i\infty} 
\Theta_\omega(z)F(-z) \, x^{-\frac{1}{2}-iz} \, dz
\] 
for $c \geq 0$. Then the right-hand side is independent of $c \geq 0$ by the assumption, 
and defines an member of $L^2((0,\infty),dx)$ by \eqref{106}. Moreover, we have
\[
\int_{0}^{x} \frac{g(u)}{\sqrt{u}} \, du = \frac{1}{2\pi} \int_{c-i\infty}^{c+i\infty} 
\Theta_\omega(z)F(-z) \, \frac{x^{-iz}}{-iz} \, dz. 
\]  
for $c>0$. On the other hand,  by Proposition \ref{a207}, we obtain 
\[
\int_{0}^{\infty} h_\omega^{\langle 1 \rangle}(xy) \, f(y) \, dy
=
\frac{1}{2\pi} \int_{c'-i\infty}^{c'+i\infty} 
\frac{\Theta_\omega(z)}{-iz}F(-z) \, x^{-\frac{1}{2}-iz} \, dz
\]
for $c' \gg 0$. Hence
\[
\frac{1}{\sqrt{x}}\int_{0}^{x} \frac{g(u)}{\sqrt{u}} \, du
=\int_{0}^{\infty} h_\omega^{\langle 1 \rangle}(xy) \, f(y) \, dy. 
\]
This implies that $g = \tilde{\mathsf H}_\omega f$. Thus $\tilde{\mathsf H}_\omega f$ is defined almost everywhere 
and belongs to $L^2((0,\infty),dx)$, since $g$ belongs to $L^2((0,\infty),dx)$.    
Moreover we obtain $g = {\mathsf H}_\omega f$ by the definition of $g$ and the latter half of the proof of Lemma \ref{lem_401}. 
Hence $\tilde{\mathsf H}_\omega f = {\mathsf H}_\omega f$, and it implies the extension of $\tilde{\mathsf H}_\omega$ to $L^2((0,\infty),dx)$. 
\hfill $\Box$

\bibliographystyle{amsplain}
\bibliography{biblio}
\end{document}